\documentclass[11pt]{amsproc}
\usepackage{amsfonts}
\usepackage{amsmath,amstext,amsbsy,amssymb}

\newtheorem{theorem}{Theorem}[section]
\newtheorem{lemma}[theorem]{Lemma}

\newtheorem{proposition}[theorem]{Proposition}
\newtheorem{corollary}[theorem]{Corollary}
\newtheorem{conjecture}[theorem]{Conjecture}
\theoremstyle{definition}
\newtheorem{definition}[theorem]{Definition}
\newtheorem{example}[theorem]{Example}

\theoremstyle{remark}
\newtheorem{remark}[theorem]{Remark}

\newcommand{\la}{\lambda}
\newcommand{\al}{\alpha}
\usepackage{pst-all}
\begin{document}

\title{Jack vertex operators and realization of Jack functions}
\author{Wuxing Cai}
\address{School of Sciences,
South China University of Technology, Guangzhou 510640, China}
\email{caiwx@scut.edu.cn}
\author{Naihuan Jing}
\address{Department of Mathematics,
   North Carolina State University,
   Raleigh, NC 27695-8205, USA and School of Sciences,
South China University of Technology, Guangzhou 510640, China}
\email{jing@math.ncsu.edu}
%\footnotetext[1]
\thanks{*Corresponding author: Naihuan Jing}
%\thanks{*Corresponding author}
%\thanks{Jing gratefully acknowledges the support from
%Max-Planck Institute f\"ur Mathematik in Bonn,
%NSF grants and Simons Foundation.}

\begin{abstract} We give an iterative method to realize general Jack
functions using vertex operators. We first show some cases of
Stanley's conjecture on positivity of
the Littlewood-Richardson coefficients, and then use this method to
give a new realization of Jack functions.
We also show in general that the images of products of Jack vertex
operators form a basis of symmetric functions. In particular this gives a new proof of linear independence for
the rectangular and marked rectangular Jack vertex operators.
Finally a generalized Frobenius formula for Jack functions is given and used
for an evaluation of Dyson integrals and even powers of Vandermonde
determinant.
\end{abstract}
 \maketitle

\section{Introduction}

Jack symmetric functions play a fundamental role in the
Calegero-Suther\-land model in statistical mechanics \cite{F} and they
also appear as extremal vectors in the Fock representation of the
Virasoro algebra \cite{MY, CJ}. From algebraic and combinatorial
viewpoints, explicit realization of Jack functions is a desirable
task and an interesting question.
 In the special case of Schur functions the explicit realization is the Jacobi-Trudi formula, which
 has a long list of
applications ranging from Schubert calculus in algebraic geometry to
various combinatorial problems. The Jacobi-Trudi formula for
rectangular Jack functions was known \cite{Ma}, and it also reveals
the connection with a special case of Dyson's conjecture for
constant terms and even powers of Vandermonde determinants
\cite{BBL}.

In this paper we first settle a long-standing problem of realizing Jack functions of all shapes
using vertex operators. The realization
is based upon our earlier work of reformulation of Jack functions
for rectangular-shapes using vertex operators \cite{MY, FF, CJ}. The key idea behind our approach
is to realize the general Jack function by a filtration of rectangular shapes and then introduce a procedure
to encode the filtration in the Fock space. We would like to remark that this new vertex operator
realization of classical symmetric functions is different from earlier realizations of Schur and Hall-Littlewood
polynomials \cite{J1, J2}.
%This relies upon
%reformulation of our earlier results for rectangular shape Jack functions.

 As an application of the new iterative vertex operator realization of Jack functions, we derive a generalized Frobenius
formula expressing a Jack function as a  linear combination of power sum symmetric functions. This formula
is new even in the case of zonal spherical functions
(in both orthogonal and symplectic cases).

As our second goal, we study the connection between the
contraction function of the Jack vertex operators for the
rectangular shapes and Dyson's constant term conjecture, and
evaluate more terms in the generalized Vandermonde determinant or
even powers of the latter. One of the key steps is to compute the
terms using our earlier work \cite{CJ} on Jack symmetric functions
of rectangular shapes. Explicitly, vertex operators provide an
effective tool to compute all coefficients of
$z_1^{-\la_1}\cdots z_i^{-\la_i}z_{s-j+1}^{\mu_j}\cdots
z_{s-1}^{\mu_2}z_s^{\mu_1} ~ (i+j\leq s)$,
in the product $\prod_{1\leq i\neq j\leq
s}(1-z_iz_j^{-1})^{\alpha}$, where
$\la=(\la_1\cdots\la_i)$, $\mu=(\mu_1,\cdots,\mu_j)$ are two partitions of equal weight.

Our third goal is to prove that the products of Jack vertex
operators form a basis in the Fock space associated to the Jack
functions. This is deeply rooted to a basic question in the
representation theory of higher level representations of affine Lie
algebras. Lepowsky-Wilson's work on Rogers-Ramanujan identities was
partly related to linear independence of some products of vertex
operators similar to our Jack vertex operators. Our case is simpler
as we can prove a general result that is equivalent to a statement
for products of vertex operators in arbitrary level representation
of the affine Lie algebra $\widehat{sl}_2$ \cite{LP, CJ}.% A detailed
%study of Jack vertex operator shows that the products of the Jack
%vertex operators labeled by partitions essentially make a basis of
%the space of symmetric functions over $\mathbb{Q}(\al)$. As a
%by-product, the $\al=1$ specialization gives a different realization
%of Schur functions of the rectangular shape with part of the last
%column removed.

Another by-product of our approach is that we can prove several
non-trivial cases of Stanley's conjecture about the
Littlewood-Richardson coefficients. We first give a formula on the
product of two Jack functions of dual shapes, which can be viewed as
a Jack case generalization of Schubert-like product in the
cohomology of the Grassmannian variety. This turns out to be a key
in computations about Stanley's conjecture. We then generalize the
result to marked rectangular (rectangle with part of the last row/column removed) case.

 This paper is organized as
follows. In section 2, we recall some basic notions about symmetric
functions and collect a few results from \cite{CJ} for later use. In
section 3 we first give a refinement and a generalization  of the
main result of \cite{CJ} and a new method to compute the
coefficients of the monomials in the generalized Vandermonde
determinant. Then we prove that the vertex operator products form a
basis in the Fock space of symmetric functions, and their basic
properties are also listed.  Section 4 treats special cases of
Stanley's conjecture, and then solves the problem of realizing Jack
functions by vertex operators in general.

 \section{The realization of Jack functions of rectangular partitions}\vskip 0.1in
   We first recall basic notions about symmetric functions.
   A partition $\lambda=(\la_1,\lambda_2,\cdots,\lambda_s)$ is a sequence of nonnegative
   integers in non-increasing order. Sometimes we also write $\la=(1^{m_1}~2^{m_2}\cdots)$,
   where $m_i=m_i(\la)$ is the multiplicity of $i$ occurring among the parts of
    $\la$. The length of $\la$, denoted as $l(\lambda)$, is the number of non-zero
    parts in $\lambda$, and the weight $|\lambda|$ is $\sum_iim_i$. It is
    convenient to denote $m(\lambda)!=m_1!m_2!\cdots$, $z_\lambda=\prod_{i\geq 1}i^{m_i}m_i!$.
     A partition $\la$ of weight $n$ is usually denoted by $\la\vdash n$. The set of all
partitions is
   denoted by $\mathcal {P}$. For two partitions
$\la=(\lambda_1,\lambda_2,\cdots)$ and $\mu=(\mu_1,\mu_2,\cdots)$ of
the same weight, we say that $\lambda\geq\mu$ if
$\lambda_1+\cdots+\lambda_i\geq\mu_1+\cdots+\mu_i$ for all
   $i$, this defines a partial order (the dominance order) in $\mathcal
   {P}$. If $m_i(\mu)\leq m_i(\la)$ for all $i$, we say that $\mu$ is {\it exponentially contained} in $\la$ or
   {\it contained by multiplicity} and we use the notation $\mu\subset'\la$.
    Partition $\la\backslash\mu$ is then defined by
    $m_i(\la\backslash\mu)=m_i(\la)-m_i(\mu)$, and we also define $\binom{m(\la)}{m(\mu)}=\prod_i\binom{m_i(\la)}{m_i(\mu)}$.
    Sometimes we also identify $\la$ with its Young or Ferrers diagram
   $Y(\la)=\{(i,j)|1\leq i\leq l(\la), 1\leq j\leq \la_i\}$ which is a left justified
   diagram made of $l(\la)$ rows of boxes (or squares as being called by some authors)
   with $\la_i$ boxes in the $i$th row. A partition of the form $\la=(r^s)$ is called a rectangular partition or a partition of rectangular shape. If the Young diagram of $\mu$ is a subset of the Young diagram of $\la$, we write $\mu\subset\la$,  define $\la-\mu$ to be the set difference $Y(\la)-Y(\mu)$. We call $\la-\mu$ a horizontal $n$-strip if there are $n$ boxes in $\la-\mu$ and with no two boxes in the same column of $\la$. The conjugate of $\la$, denoted as $\la'$, is a partition whose diagram is the transpose of the diagram of $\la$, thus $\la'_i$ is the number elements in the set $\{j|\la_j\geq i\}$. For a box $s=(i,j)\in\la$, the
lower hook-length $h_*^\lambda(s)$ is defined to be
   $\alpha^{-1}(\lambda_i-j)+(\lambda'_j-i+1)$ in the rational field $\mathbb{Q}(\al)$ of indeterminant $\al$, and the upper hook-length
$h^*_\lambda(s)=\alpha^{-1}(\lambda_i-j+1)+(\lambda'_j-i)$.
\begin{example}Fig.1 (a) is partition $\la=(5,3,3,1)$, Fig.1 (b) is the diagram of partition $\la'$. For $a=(1,1), b=(2,2)\in Y(\la)$, we have $h_*^\la(a)=4\al^{-1}+4$, $h^*_\la(b)=2\al^{-1}+1$. As shown in Fig.1 (c), the coefficient of $\al^{-1}$ (resp. constant) is the number of boxes the solid line passes horizontally (resp. vertically).
\end{example}
 \setlength{\unitlength}{0.5cm}
\begin{picture}(1,1)
%\la=(6,3,3,1)
%horizontal
 \put(0,0){\line(1,0){5}}
 \put(0,-1){\line(1,0){5}}
 \put(0,-2){\line(1,0){3}}
 \put(0,-3){\line(1,0){3}}
 \put(0,-4){\line(1,0){1}}
 %hverticals
\put(0,0){\line(0,-1){4}}
\put(1,0){\line(0,-1){4}}
\put(2,0){\line(0,-1){3}}
\put(3,0){\line(0,-1){3}}
\put(4,0){\line(0,-1){1}}
\put(5,0){\line(0,-1){1}}
\put(0,-5.9){Fig.1 (a)}
%The upside-down (2,1)in the second corner
\put(2.3,-1.7){$-$}
\put(1.3,-2.7){$-$}
\put(2.3,-2.7){$-$}

%\la'=(4,3,3,1,1)
%horizon
\put(9,0){\line(1,0){4}}
\put(9,-1){\line(1,0){4}}
\put(9,-2){\line(1,0){3}}
\put(9,-3){\line(1,0){3}}
\put(9,-4){\line(1,0){1}}
\put(9,-5){\line(1,0){1}}
%vertical
\put(9,0){\line(0,-1){5}}
\put(10,0){\line(0,-1){5}}
\put(11,0){\line(0,-1){3}}
\put(12,0){\line(0,-1){3}}
\put(13,0){\line(0,-1){1}}

\put(9,-5.9){Fig.1 (b)}

\put(10.3,-1.8){*}
\put(11.3,-1.8){*}
\put(10.3,-2.8){*}
\put(11.3,-2.8){*}
%h_*,h^*
%horizontal
 \put(18,0){\line(1,0){5}}
 \put(18,-1){\line(1,0){5}}
 \put(18,-2){\line(1,0){3}}
 \put(18,-3){\line(1,0){3}}
 \put(18,-4){\line(1,0){1}}
 %hverticals
\put(18,0){\line(0,-1){4}}
\put(19,0){\line(0,-1){4}}
\put(20,0){\line(0,-1){3}}
\put(21,0){\line(0,-1){3}}
\put(22,0){\line(0,-1){1}}
\put(23,0){\line(0,-1){1}}

\put(18.1,-0.4){$a$}

 \psline[linestyle=dashed,dash=0.05 0.05](9.3,-0.25)(9.5,-0.25)
\put(19,-0.5){\vector(1,0){4}}
\put(18.5,-0.5){\vector(0,-1){3.5}}

\put(19.1,-1.6){$b$}
 \psline[linestyle=dashed,dash=0.05 0.05](9.75,-0.75)(9.75,-1)
\put(19.5,-1.5){\vector(1,0){1.5}}
\put(19.5,-2){\vector(0,-1){1}}
\put(18,-5.9){Fig.1 (c)}
\end{picture}

\vskip 3.7cm

We sometimes write partition $\la=(a_1^{n_1}a_2^{n_2}\cdots a_s^{n_s})$, with $m_{a_i}(\la)=n_i>0$, and $a_1>a_2>\cdots>a_s>0$. For each $i=1,2,\cdots s$, we replace the part $a_i$ of $\la$ by $a_{i+1}$, setting $a_{s+1}=0$, therefore obtaining a new partition, denoted as $\la(i)$. Equivalently, $$\la(i)=(a_1^{n_1}a_2^{n_2}\cdots a_{i-1}^{n_{i-1}}a_{i+1}^{n_i+n_{i+1}}a_{i+2}^{n_{i+2}}\cdots a_s^{n_s}).$$ Then we define the $i$th corner of $\la$ as $C_i(\la)=\la-\la(i)$. $C_i(\la)$ is of the same shape as the rectangle $((a_i-a_{i+1})^{n_i})$. $\la$ has $s$ corners, we write $n_c(\la)=s$ as the number of corners of $\la$. $n_c(\la)=1$ if and only if $\la$ is a rectangular partition and in this case the unique corner is $\la$ itself. For a partition $\nu=(\nu_1,\cdots,\nu_l)$ with $\nu_1\leq a_i-a_{i+1}, l\leq n_i$,  we define the upside-down $\nu$ of the $C_i(\la)$ as the union of $B_j$ for $j=1,2,\cdots,l$, where $B_j$ are the rightmost $\nu_j$ boxes on the $(n_i+1-j)$th row of $C_i(\la)$. For example, we can write $\la'=(4,3,3,1,1)$ as $(4^13^21^2)$, $\la'(2)=(4^11^4)$ and the second corner of $\la'$ is $C_2(\la')=\la'-\la'(2)=\{(2,2),(2,3),(3,2),(3,3)\}$, i.e. the boxes marked by stars. An upside-down $(2,1)$ of the second corner of $\la$ is marked by the minus sign in Fig.1 (a).
   \begin{definition}\label{bottomnotation}
For a subset $S$ of the Young diagram of $\la$, define $h_*^\la(S)$ as
$\prod_{s\in S}h_*^\la(s)$, and similarly for $h_\la^*(S)$. For a
pair of partitions $\mu\subset\la$, we say that a square $s$ is
based if $s$ is in a column which contains at least one square
of $\la-\mu$, and it is un-based otherwise. We denote $\mu_b$%=\mu_b^{\subset\la}$
(resp. $\la_b$) as the set of the based squares of $\mu$(resp.
$\la$), and $\mu_u$ (resp. $\la_u$) as the set of the un-based
squares of $\mu$ (resp. $\la$).
\end{definition}

%Knop and Sahi proved that
%\begin{theorem}\cite{KS}\label{T:KSMac}
%Let
%$J_\lambda(\alpha)=\sum_{\mu\leq\lambda}v_{\lambda\nu}(\alpha)m_\mu,$
%then we have $\frac{v_{\la\mu}(\alpha)}{m(\mu)!}\in
%\mathbb{Z}_{\geq0}[\al].$
%\end{theorem}
%A direct consequence of this is
%\begin{corollary}
%$$C_{\mu\nu}^\la(\al)=\langle J_\mu(\alpha)
%J_\nu(\alpha),J_\lambda(\alpha)\rangle\in\mathbb{Z}[\al].$$
%\end{corollary}
%
%In fact, Stanley conjecture that
%$C_{\mu\nu}^\la(\al)\in\mathbb{Z}_{\geq 0}[\al]$ in \cite{S}.

   For a set (finite or infinite) of variables $z_1, z_2, \cdots$
   and a partition $\lambda$, we define $z^\lambda=z_1^{\lambda_1}
   z_2^{\lambda_2}\cdots$. For a Laurent polynomial $F(z,w)$ in
   variables
   $z_1,z_2,\cdots,w_1,w_2,\cdots$, we denote the coefficient of monomial $z^\lambda/w^\mu$ as
   $C_{z^\la/w^\mu}.F(z,w)$.

The ring $\Lambda$ of symmetric functions in variable $x_i$'s is a $\mathbb Z$-module
with basis $m_{\la}=(m(\la)!)^{-1}\sum x^{\la_1}_{i_1}\cdots x_{i_k}^{\la_k}$, where the multi-indices $(i_1,\cdots, i_k)$
run through $(\mathbb{Z}_{>0})^k$.
We define $p_n=m_{(n)}=\sum_{i\geq1}x_i^n$. The power sum symmetric functions
$p_{\la}=p_{\la_1}\cdots p_{\la_k},$ $\la\in\mathcal{P}$, form a basis in
$\Lambda_\mathbb{Q}=\Lambda\otimes_\mathbb{Z}\mathbb{Q}$.

Let $F=\mathbb{Q}(\alpha)$, the field of rational functions in
indeterminate $\alpha$. For two partitions $\la, \mu \in \mathcal P$
the Jack scalar product on $\Lambda_F=\Lambda\otimes_\mathbb{Z}F$
is given by
\begin{align} \label{def}
\langle p_{\la}, p_{\mu}\rangle=\delta_{\la,
\mu}\alpha^{-l(\la)}z_\lambda,
\end{align}
where $\delta$ is the Kronecker symbol. Here for convenience we use
$\alpha^{-1}$ instead of $\alpha$ in light of Sections 2 and 3.

The Jack symmetric functions or Jack polynomials \cite{Ja} $P_{\lambda}(\alpha^{-1})=P_\la(x_1, x_2,\cdots; \al^{-1})$ for
$\la\in\mathcal{P}$ are defined by the following conditions:
\begin{align*}
&P_{\la}(\alpha^{-1})=\sum_{\la\geq\mu}c_{\la
\mu}(\alpha^{-1})m_{\mu},\\
&\langle
P_{\la}(\alpha^{-1}),P_{\mu}(\alpha^{-1})\rangle=0~~\mbox{for}~\la\neq\mu,
\end{align*}
where $c_{\la \mu}(\alpha^{-1})\in F$ ($\la, \mu \in \mathcal P$)
and $c_{\la \la}(\alpha^{-1})=1$. Note that symmetric functions $m_\la, p_\la, P_\la(\al^{-1})$, and the following $Q_\la(\al^{-1}), J_\la(\al^{-1})$ are all in variable $x_1,x_2,\cdots$. The existence of Jack symmetric
functions was proved in \cite{M}. Further properties of Jack functions are studied in
\cite{S}, \cite{KS} and \cite{H}. In this work we will also need two other normalizations.
One is the dual Jack functions $Q_{\lambda}(\alpha^{-1})=\langle
P_{\lambda}(\alpha^{-1}),
P_{\lambda}(\alpha^{-1})\rangle^{-1}P_{\lambda}(\alpha^{-1})$. Another is the integral Jack
function
$J_\lambda(\alpha^{-1})$, which is defined by $J_\lambda(\alpha^{-1})=|\la|!(c_{\la,(1^{|\la|})})^{-1}P_\la(\al^{-1})$.
Let us recall Pieri formula of Jack functions for later use.
\begin{theorem}\cite{S}\label{T:Pieriformula}
For partitions $\mu, \la$ and non-negative integer $n$, the scalar product $\langle J_\mu(\al^{-1})J_n(\al^{-1}), J_\la(\al^{-1})\rangle\neq0$ if and only if $\la-\mu$ is a horizontal $n$-strip. And in this case
$$\langle J_\mu(\al^{-1})J_n(\al^{-1}), J_\la(\al^{-1})\rangle=n!\al^{-n} h^*_\la(\la_u)h_*^\la(\la_b) h^*_\mu(\mu_b )h_*^\mu(\mu_u).$$
\end{theorem}

For later use, we would like to introduce some terminology here. For an operator $S$ on symmetric functions, if $\langle J_\mu, S.J_\la\rangle$ is non-zero, we say that $S$ can remove $\la-\mu$ from $\la$; otherwise we say that $S$ can not remove $\la-\mu$ from $\la$. For example, $J_n(\al^{-1})^*$ can remove a horizontal $n$-strip from $\la$ if $\la_1\geq n$. Furthermore if $\la$ is of rectangular shape, then $J_n(\al^{-1})^*$ can only remove the $n$ boxes on the right side of last row of $\la$.

 The generalized complete symmetric functions of $\Lambda_F$ are
defined by
\begin{equation}
q_{\la}(\alpha^{-1})=Q_{\la_1}(\alpha^{-1})Q_{\la_2}(\alpha^{-1})\cdots
Q_{\la_l}(\alpha^{-1}),
\end{equation}
where $Q_n(\alpha^{-1})=Q_{(n)}(\alpha^{-1})$ has a
generating function given by
%\begin{equation}\label{E:homogeneous1}
%Q_n(\alpha^{-1})=\sum_{\lambda\vdash
%n}\alpha^{l(\lambda)}z_\lambda^{-1}p_{\lambda}.
%\end{equation}
\begin{equation}\label{generatingfunctionofonerowJack}
exp\Big(\sum_{n=1}^\infty
\frac{z^n}{n}{\alpha}p_{n}\Big)=\sum_{n\in\mathbb{Z}}
Q_n(\al^{-1})z^{n}.
\end{equation}
Thus $Q_0(\al^{-1})=1$ and $Q_n(\al^{-1})=0$ for $n<0$.
We list some of the useful properties of $q_{\lambda}$ and $J_\lambda$
as follows.

\begin{lemma} \cite{S}\label{L:triangular}
For any partition $\lambda$, $\nu$, one has
\begin{align*}
&\langle q_{\la}(\alpha^{-1}),m_{\nu}\rangle=\delta_{\la,\nu},\\
&Q_{\lambda}(\alpha^{-1})=q_{\la}(\alpha^{-1})+\sum_{\mu>\lambda}a_{\lambda\mu}(\al^{-1})q_{\mu}(\alpha^{-1}),\\
&J_\lambda(\al^{-1})=h^\la_*(\la)P_\la(\al^{-1}),\\
&J_\lambda(\al^{-1})=h_\la^*(\la)Q_\la(\al^{-1}),
\end{align*}
%$$Q_{\lambda}(\alpha^{-1})=\sum_{\mu\geq\lambda}d'_{\lambda\mu}q_{\mu}(\alpha^{-1}),$$and
where %$d'_{\lambda,\mu}\in F$,with $c'_{\lambda\mu}=c_{\mu\lambda}$
$a_{\lambda\mu}(\al^{-1})\in F=\mathbb{Q}(\al)$.
\end{lemma}

\subsection{Fock space $V$ and rectangular Jack functions}
 Let us recall the realization of rectangular Jack functions on
 the Fock space of the rank one vertex operator algebra.
  Let
  $$V=Sym(h_{-1},h_{-2},\cdots)\otimes
\mathbb{C}[\frac12\mathbb{Z}],$$
where
$V_0=Sym(h_{-1},h_{-2},\cdots)$ is the complex symmetric algebra
 generated by the variable $h_{-n} (n\in\mathbb N)$ and
  $\mathbb{C}[\frac12\mathbb{Z}]$ is the group algebra of $\frac{1}{2}\mathbb Z$ with
  generators
  $\{e^{nh}|n\in\frac{1}{2}\mathbb{Z}\}$ with multiplication given by
  $e^{mh}e^{nh}=e^{(m+n)h}$. The underlying Heisenberg algebra is
  the infinite dimensional Lie algebra generated by $h_n,
  n\in\mathbb Z-\{0\}$ subject to relations
  \begin{equation}
  [h_m, h_n]=m\alpha^{-1}\delta_{m, -n}I.
  \end{equation}
For convenience, we set $h_0=1\in V_0$. The associated scalar product of $V$ is given by

\begin{align}
\label{scalarproduct}
  \langle h_{-\lambda}
\otimes e^{mh},h_{-\mu} \otimes
  e^{nh}\rangle=z_\lambda {\alpha}^{-l(\lambda)} \delta_{\lambda,\mu}\delta_{mn},
\end{align}
 where for partition $\la=(\la_1,\la_2,\cdots)$, $h_{-\la}$ is defined to be $h_{-\la_1}h_{-\la_2}\cdots$.

For an operator $S$ on $V$ (or $\Lambda_F$), the conjugate of $S$, $S^*$ is defined
by $\langle S.u, v\rangle=\langle u, S^*.v\rangle$, $u,v\in V$.
For a symmetric function $f$, we also denote the multiplication operator
$g\rightarrow fg$ by the same symbol $f$. It can be shown that $h_{n}^*=h_{-n}$ for $n\in\mathbb{Z}$.

Note that $B=\{h_{-\la}\otimes e^{nh}|\la\in\mathcal{P}, n\in \frac{1}{2}\mathbb{Z}\}$ is a basis of $V$, we can define a linear map $T$ from $V$ to $\Lambda_F$ by
$$T(h_{-\lambda}\otimes
e^{nh})=p_\lambda.$$
 Comparing formulae (\ref{def}) and (\ref{scalarproduct}), we know that $T$ preserve the scalar product, i.e.
 $$\langle T(u), T(v)\rangle=\langle u, v\rangle,$$
 for every pair $u,v\in V$.

 We then define the vertex operators
\begin{align} &X(z)=exp\Big(\sum_{n=1}^\infty
\frac{z^n}{n}{\alpha}h_{-n}\Big) A
  exp\Big(\sum_{n=1}^\infty \frac{z^{-n}}{-n}2{\alpha}h_n\Big)=\sum_{n\in\mathbb{Z}}
  X_nz^{-n},\\\label{D:halfvertex}
&Y(z)=exp\Big(\sum_{n=1}^\infty
\frac{z^n}{n}{\alpha}h_{-n}\Big)=\sum_{n\in\mathbb{Z}} Y_nz^{-n},\\
&Y^*(z)=exp\Big(\sum_{n=
1}^{\infty}\frac{z^{n}}{n}{\alpha}h_n\Big)=\sum Y_n^*z^{-n},
\end{align}
where
\begin{eqnarray*}
      & h_n.v\otimes e^{sh}= ({\alpha}^{-1}n\frac{\partial}{\partial h_{-n}}v)\otimes e^{sh} &\mbox{if}\ \  n>0,\\
      & h_n.v\otimes e^{sh}=(h_{n}v)\otimes e^{sh} & \mbox{if}\ \ n<0,
\end{eqnarray*}
and %A=exp(2{\alpha}lnz\partial_h+h)
\begin{equation}
A.v\otimes
e^{sh}=z^{(s+\frac{1}{2})2{\alpha}}v\otimes e^{(s+1)h}.
\end{equation}
 \vskip 0.1in
 \begin{remark}\label{R:one-row-Jack}
 Comparing (\ref{generatingfunctionofonerowJack}) with (\ref{D:halfvertex}), we know that $T(Y_{-n})=Q_n(\al^{-1})$ for $n\in\mathbb{Z}$. Moreover, because $h_{-n}^*=h_n$, we know that $Y_n^*$ is actually the conjugate of $Y_n$ for $n\in\mathbb{Z}$.

 \end{remark}
The following result  provides a realization of Jack functions of
rectangular shapes. For a partition $\la=(\la_1,\cdots,\la_l)$ of length $l$, we denote $X_{-\la}=X_{-\la_1}X_{-\la_2}\cdots X_{-\la_l}$, and $Y_{-\la}=Y_{-\la_1}Y_{-\la_2}\cdots Y_{-\la_s}$.

\begin{theorem} \cite{CJ} \label{L:squareJack}  For partition $\lambda=((k+1)^s,(k)^t)$ with $k\in
\mathbb{Z}_{\geq 0}$, $s\in \mathbb{Z}_{>0}$, $t\in \{0,1\}$ and $\al\in\mathbb{Z}_{>0}$, we
have
$$T\Big(X_{-\lambda}e^{-(s+t)h/2}\Big)=c_{s,t}(\alpha)Q_{\lambda}({\alpha}^{-1}),$$
where $c_{s,t}(\alpha)$ is a function of $\alpha$, and
$c_{s,t}(1)=(-1)^{s(s-1+2t)/2}s!$.
\end{theorem}

We also need the following proposition for later use.
\begin{proposition} \cite{CJ}\label{P:Xlambdaqmu}
For positive integer $\al$ and partitions $\la,\mu$ of the same
weight, we have
\begin{equation}\label{F:contraction}
\langle
X_{-\lambda}.e^{-l(\lambda)h/2},Y_{-\mu}.e^{l(\lambda)h/2}\rangle=
(-1)^{s(s-1)\alpha/2}C_{z^{\lambda}/w^{\mu}}.H_\alpha(Z_s,W_t),
\end{equation}
where $H_{{\alpha}}(Z_s, W_t)=\prod_{s\geq i\neq j\geq
1}(1-z_jz_i^{-1})^{{\alpha}}\prod_{j=1}^s
\prod_{i=1}^t(1-z_jw_i^{-1})^{-{\alpha}}$. Moreover, (\ref{F:contraction})
is zero if $\mu_1>\la_1$.
\end{proposition}
\section{A basis for symmetric functions}
\subsection{Further results on Jack vertex
operator} \vskip 0.1in
 Suppose $\alpha\in\mathbb N$. A special case of Dyson's constant
term conjecture, proved by Wilson \cite{Wi} and Gunson\cite{Gu}
independently, says that the constant term of
$$\Delta_s^\alpha=\prod_{1\leq i\neq j\leq
s}(1-\frac{D_i}{D_j})^\alpha$$ is $(s\alpha)!(\alpha!)^{-s}$, thus
we have the following refinement to Theorem \ref{L:squareJack}.
\begin{corollary}\label{C:explicitcoefficient}
In Theorem \ref{L:squareJack}, we have $$c_s(\alpha)=c_{s,0}(\al)=(-1)^{\alpha
s(s-1)/2}(s\alpha)!(\alpha!)^{-s}$$ in the case of $t=0$.
\end{corollary}
\begin{remark}If $s\geq2$, then $c_s(\alpha)$ is not a rational function of $\al$.
\end{remark}
\begin{proof} Recall that we have \cite{CJ}
\begin{align}\label{F:Xlambda}
&T(X_{-\lambda_1}X_{-\lambda_2}\cdots X_{-\lambda_s}.e^{nh})\\\nonumber
&=(-1)^{\alpha s(s-1)/2}\prod_{1\leq i\neq j\leq
s}(1-\frac{D_i}{D_j})^\alpha.Q_{\lambda_1}({\alpha}^{-1})\cdots
Q_{\lambda_s}({\alpha}^{-1})\otimes e^{(n+s)h},
\end{align}
where $D_i$ is the lowering operator defined as
$D_i.Q_{\lambda_1}({\alpha}^{-1})\cdots
Q_{\lambda_s}({\alpha}^{-1})=Q_{\lambda_1}({\alpha}^{-1})\cdots
Q_{\lambda_i-1}({\alpha}^{-1})\cdots Q_{\lambda_s}({\alpha}^{-1})$.
Note that for rectangular $\la$, only constant term contributes to
$q_\la$, the statement then follows from Dyson's constant term
conjecture.
\end{proof}

From the proof we also have the following equation (\ref{squareJackexpression}) which
links even power
of Vandernomde determinant to rectangular Jack functions:

\begin{equation}\label{squareJackexpression}
 \Delta_s^\al.q_{(k^s)}(\alpha^{-1})=\frac{(s\alpha)!}{(\alpha!)^s}Q_{(k^s)}(\alpha^{-1}).
\end{equation}

In the context of Theorem \ref{L:squareJack}, there is also
a combinatorial formula in terms of $p_\mu$'s for the products of
vertex operators (Theorem 3.2 in \cite{CJ}). Combining this with
Corollary \ref{C:explicitcoefficient}, we get a new expression for
rectangular Jack functions. For clarity and latter use, let us first
give the following:
\begin{definition}\label{D:co.ofrectgularJack}
For rectangular partition $R=(k^s)$, positive integer $\al$ and
$\mu\vdash ks$, define
\begin{align}
&g_{R,\mu}(\al^{-1})=(-1)^{\al
s(s-1)/2}\frac{(\al!)^s}{(s\al)!(-2)^{l(\mu)}}
\sum_{\underline{\mu},\underline{\nu}}\prod_{i=1}^s
\frac{(-2{\alpha})^{l(\nu^i)}}{z_{\nu^i}}
\binom{m(\mu^{i-1})}{m(\mu^i \backslash \nu^i)},
\end{align}
where the sum is over sequences of partitions $
\underline{\mu}=(\mu^1,\mu^2,\cdots,\mu^s)$ and
$\underline{\nu}=(\nu^1, \nu^2,\cdots,\nu^s)$ such that $\mu^0=(0)$,
$\mu^s=\mu$, $\nu^i\subset'\mu^i$, $|\mu^i|=i(k+(s-i)\al)$, and
$\mu^i\backslash\nu^i\subset'\mu^{i-1}$.
\end{definition}
\begin{corollary}\label{C:formula1}
For positive integers $k,s,\al$, we have
\begin{equation}\label{F:formula1}
Q_{(k^s)}(\al^{-1})=\sum_{\mu\vdash ks}g_{(k^s),\mu}(\al^{-1})p_\mu.
\end{equation}
\end{corollary}

Note that we can expand Jack functions in terms of $p_\la$'s by
other approaches, thus Corollary \ref{C:formula1} also gives various different
combinatorial identities correspondingly. As a simple example, we set
$(k^s)=(1^2)$, and get the following nontrivial identity by considering coefficients of $p_1^2$ on two sides of
identity (\ref{F:formula1}).

\begin{corollary} For positive integer $\alpha$ one has
\begin{equation*}
\sum_{i=0}^2\frac{(-2\al)^{2-i}}{(2-i)!}
\sum_{(1^i)\subset'\nu\vdash1+\al}\frac{(-2\al)^{l(\nu)}}{z_\nu}\binom{m(\nu)}{m(1^i)}=\frac{(2\al)!}{(\al
!)^2}\frac{4\al^2}{\al+1}(-1)^\al.
\end{equation*}
\end{corollary}

Since we realized the Jack functions  of rectangular shapes in $t=0$ case of
Theorem \ref{L:squareJack}, we can further generalize this result
to the case of a rectangular tableau plus one row.
 \begin{corollary} For partition $\lambda=((k+1)^s)$ with $k\in
\mathbb{Z}_{\geq 0}$, $s\in \mathbb{Z}_{>0}$, and integer $n$ with $0\leq n\leq k+1$, we have
$$T(Y_{-n}^*X_{-\la}.e^{-sh/2})=d_{s,k}(\alpha)Q_{((k+1)^{s-1},k+1-n)}({\alpha}^{-1})$$
where $d_{s,k}(\alpha)\neq 0$ is a function of $\alpha$.
\end{corollary}
\begin{proof} By Remark \ref{R:one-row-Jack}, it is immediate that $$T(Y_{-n}^*X_{-\la}.e^{-sh/2})=Q_{-n}(\al^{-1})^*.(c_{s,0}(\al)Q_\la(\al^{-1})),$$ where we used Theorem \ref{L:squareJack} and its notation. By Pieri formula Theorem \ref{T:Pieriformula},  $\langle Q_{-n}(\al^{-1})^*.Q_\la(\al^{-1}),Q_\mu(\al^{-1})\rangle=\langle Q_\la(\al^{-1}),Q_{-n}(\al^{-1})Q_\mu(\al^{-1})\rangle$ is non-zero if only if $\la-\mu$ is a horizontal $n$-strip. But $\la=((k+1)^s)$ is of rectangular shape, there is a unique $\mu$ satisfies this property, i.e. $\mu=((k+1)^{s-1},k+1-n)$.
\end{proof}

\vskip 0.2in
\subsection{The even power of Vandermonde determinant}
The lowering operator $\Delta_s^\alpha=\prod_{1\leq i\neq j\leq
s}(1-\frac{D_i}{D_j})^\alpha$ is essentially even power of
Vandernomde determinant in the sense that $$\prod_{1\leq j<i\leq
s}(D_i-D_j)^{2\al}=(D_1\cdots
D_s)^{\al(s-1)}(-1)^{s(s-1)\al/2}\Delta_s^\alpha.$$ Thus the
following corollary can be used to study the generalized Vandermonde
determinant. Recall that by Lemma \ref{L:triangular}, the transition matrixes between $q_\la(\al^{-1})$'s and $Q_\la(\al^{-1})$'s are triangular matrixes, if we arrange them by the partitions with a total order which is compatible with the dominance order.
\begin{corollary}\label{C:co-of-delta}
Set
$Q_{(k^s)}(\al^{-1})=\sum_{\mu\geq(k^s)}a_\mu(\al^{-1})q_\mu(\al^{-1})$.
For $\beta=(b_1,\cdots,b_s)\in\mathbb{Z}^s$ such that
$b_1+\cdots+b_s=0$ and $b_1\geq b_2\cdots \geq b_s>-k$, we have
\begin{equation}
C_{D_1^{b_1}\cdots
D_s^{b_s}}.\Delta_s^\al=\frac{m(\mu_\beta)!}{s!}\frac{(s\al)!}{(\al!)^s}a_{\mu_\beta}(\al^{-1}),
\end{equation}
where $\mu_\beta=(k^s)+\beta$ is a partition.
\end{corollary}
\begin{proof}
Consider coefficient of $q_{\mu_\beta}(\al^{-1})$ on two sides of
(\ref{squareJackexpression}), and notice that including
$D_1^{b_1}\cdots D_s^{b_s}$, there are $s!/m(\mu_\beta)!$ terms
contributing to $q_{\mu_\beta}(\al^{-1})$ when $\Delta_s^\al$ acts on
$q_{(k^s)}(\alpha^{-1})$.
\end{proof}
\begin{remark} Dyson's constant term corresponds to the
special case that $\beta=0\in\mathbb{Z}^s$. Note that we have an
iterative formula \cite{CJ1,LLM} for the coefficients
$a_\mu(\al^{-1})$'s. Hence these coefficients can be evaluated
effectively using vertex operator calculus. The following gives some examples on this.
\end{remark}
\begin{proposition}\label{P:alambdai} For positive integers
$i,s,\al$ with $2i\leq s$, the coefficient of
$D_1^{-1}\cdots D_i^{-1}D_{s-i+1}D_{s-i+2}\cdots D_s$ in Laurent
polynomial $\Delta_s^\al=\prod_{1\leq i\neq j\leq
s}(1-\frac{D_i}{D_j})^\alpha$ is
\begin{align}\label{colambdai}
&C_{D_1^{-1}\cdots D_i^{-1}D_{s-i+1}D_{s-i+2}\cdots
D_s}.\Delta_s^\alpha=\frac{(s\alpha)!}{(\alpha!)^s}i!(-\alpha)^i\prod_{j=1}^i(1+(s-j)\alpha)^{-1}.
\end{align}
And we also have
\begin{align*}
C_{D_1^{-2}D_s^2}.\Delta_s^\alpha&=\frac{(s\alpha)!}{(\alpha!)^s}\cdot
\frac{2\alpha^2-(1+(s-1)\alpha)(1+(s-2)\alpha)\alpha}{(1+(s-1)\alpha)(1+(s-2)\alpha)(2+(s-1)\alpha)},\\
C_{D_1^{-1} D_2^{-1}D_s^2}.\Delta_s^\alpha
&=\frac{(s\alpha)!}{(\alpha!)^s}\frac{2(2+(s-1)\alpha)\al^2}{(1+(s-1)\alpha)(1+(s-2)\alpha)(3+(2s-3)\alpha)}.
\end{align*}
\end{proposition}
\begin{remark}
 In \cite{K}, Kadell evaluated the coefficients of $z_1z_s^{-1}$,
$z_1z_2z_{s-1}^{-1}z_s^{-1}$ and $z_1z_2z_s^{-2}$ in Dyson's Laurent
polynomial $\prod_{1\leq i\neq j\leq s}(1-z_iz_j^{-1})^{\alpha_i}$.
Here we consider the special case when the exponents $\alpha_i$ are
identical, but we can compute the coefficient of $z_1^{-\la_1}\cdots
z_i^{-\la_i}z_{s-j+1}^{\mu_j}\cdots z_{s-1}^{\mu_2}z_s^{\mu_1}$ for
any pair of partitions $\la=(\la_1\cdots\la_i)$,
$\mu=(\mu_1,\cdots,\mu_j)$, with $|\la|=|\mu|$ and $i+j\leq s$. Moreover
the steps of computation depend on $\la$ and $\mu$ but not $\alpha$
and $s$.
\end{remark}
\subsection{The independence of the Vertex operator products}
For a partition $\la$, denote
$T(X_{-\lambda}.e^{-l(\lambda)h/2})$ as $X_{-\la}(\al^{-1})$, and set
$$X'_{-\lambda}=X'_{-\lambda}(\al^{-1})=c_{l(\la)}(\alpha)^{-1}X_{-\lambda}(\al^{-1}),$$
where
$c_s(\alpha)=\frac{(s\alpha)!}{(\alpha!)^s}(-1)^{s(s-1)\alpha/2}$ is
defined in Corollary \ref{C:explicitcoefficient}.  Note that in general $c_s(\al)$ is \emph{not} a rational function of $\al$, and that is why we factor it out from $X_{-\lambda}(\al^{-1})$. We will prove
these $X'_\la$'s form a basis for the space of symmetric functions
over $\mathbb{Q}(\al)$. First, Corollary \ref{C:explicitcoefficient}
implies that $X'_{-\lambda}$ is in $\Lambda_F$ for rectangular
partition $\la$. In fact, this is true for general partitions.
\begin{lemma}
 For any partition $\lambda$, $X'_{-\lambda}$ is a symmetric function
over $\mathbb{Q}(\al)$, i.e. $X'_{-\lambda}(\al^{-1})\in \Lambda_F$.
\end{lemma}
\begin{proof}
In (\ref{F:Xlambda}), the coefficient of the term
$q_{\mu}(\alpha^{-1})$ for each $\mu$ is $c_s(\alpha)$ multiplied by
some rational function (which depends on $\mu$) of $\al$ by
Corollary \ref{C:co-of-delta}, thus we can factor $c_s(\alpha)$ out
and rewrite $T(X_{-\lambda}.e^{-l(\lambda)h/2})$ in the desired form.
\end{proof}

Second, we can prove that the $X_{-\la}(1)$'s are linearly
independent with the following detailed analysis about the
 specification of $\al=1$ in Proposition \ref{P:Xlambdaqmu}.

\begin{lemma}\label{L:H1analysis}
For $$H=H_1(Z_s,W_t)=\prod_{1\leq i\neq j\leq
s}(1-z_iz_j^{-1})\prod_{i=1}^s\prod_{j=1}^t(1-z_iw_j^{-1})^{-1}$$
let $\lambda=(\lambda_1,\cdots,\lambda_s)$ be a partition of length
$s$, $\mu=(\mu_1,\cdots,\mu_t)$ be a partition of length $\leq t$.
%$z^\lambda/w^\mu=z_1^{\lambda_1}\cdots
%z_s^{\lambda_s}/(w_1^{\mu_1}\cdots w_t^{\mu_t})$.
Then $C_{z^{\lambda}/w^\mu}.H\neq0$ implies $\lambda\geq\mu$.
Moreover $C_{z^\lambda/w^\lambda}.H$ is a positive integer.
\end{lemma}

\begin{proof} We use induction on $s$. The case of $s=1$ or $2$ is trivial by Proposition \ref{P:Xlambdaqmu}.
 Assume that it is
true in the case of $s-1$. Then in the case of $s$, we have
$H=\sum_{l=1}^s H_{1,l}$, where
\begin{align*}
H_{1,l}=&(1-z_lw_1^{-1})^{-1}\cdot\prod_{k\neq
l}(1-z_lz_k^{-1})\cdot\prod_{j=2}^t(1-z_lw_j^{-1})^{-1}\\
&\cdot H_1(z_1,\cdots,\widehat{z_l}, \cdots, z_s; w_2,\cdots, w_t),
\end{align*}
where the hat $~\widehat{}~~$ means omission. By induction
assumption, we know that any term $z_1^{a_1}\cdots
\widehat{z_l^{a_l}}\cdots z_s^{a_s}/(w_2^{b_2}\cdots w_t^{b_t})$
from
$$H_{1,l,1}=H_1(z_1,\cdots,\widehat{z_l},\cdots,
z_s;w_2,\cdots,w_t)$$ has the property that
$a_{j_1}+\cdots+a_{j_i}\geq b_2+\cdots+b_{i+1}$ for $i$ distinct
numbers $j_1,\cdots,j_i$ from the set
$\{1,\cdots,\widehat{l},\cdots,s\}$. The term $z_1^{c_1}\cdots
 z_s^{c_s}/(w_1^{d_1}\cdots w_t^{d_t})$ from $H_{1,l}/H_{1,l,1}$ has the property that
 $c_l+c_{j_1}+\cdots+c_{j_i}\geq d_1+\cdots+d_{i+1}$. Multiplying
 two terms together, the term
 $$z_1^{a_1+c_1}\cdots
z_l^{c_l}\cdots z_s^{a_s+c_s}/(w_1^{d_1}w_2^{b_2+d_2}\cdots
w_t^{b_t+d_t})$$
from $H_{1,l}$
has the property that
$a_{j_1}+\cdots+a_{j_i}+c_l+c_{j_1}+\cdots+c_{j_i}\geq
b_2+\cdots+b_{i+1}+d_1+\cdots+d_{i+1}$. Thus if it is written as
$z^{\lambda}/w^{\mu}$ for two partitions $\lambda,\mu$ , one should
have $\lambda_1+\cdots+\lambda_{i+1}\geq\mu_1+\cdots+\mu_{i+1}$.
This finishes the proof of the first statement.

For the second statement, let
$\lambda_1=\cdots=\lambda_a>\lambda_{a+1}$, then the term
$z^{\lambda}/w^{\lambda}$ can only come from $H_{1,l}$ with
$l=1,\cdots,a$. The only way it appears in $H_{1,l}$ is when we have
$z_l^{\lambda_l}/w_1^{\lambda_1}$ from $(1-z_lw_1^{-1})^{-1}$, with
coefficient $1$, and then $z_1^{\lambda_1}\cdots
\widehat{z_l^{\lambda_l}}\cdots
z_s^{\lambda_s}/(w_2^{\lambda_2}\cdots w_t^{\lambda_s})$ from
$H_{1,l,1}$, with coefficient being a positive integer by induction,
we then multiply them together to get the result.
\end{proof}

\begin{corollary}
For positive integer $n$, $\{X_{-\lambda}(1)|\lambda\vdash n\}$ is
linearly independent.
\end{corollary}
\begin{proof} By the previous lemma, when expended into combinations of
monomials, $X_{-\lambda}(1)$ contains only terms $m_\mu$ with
$\mu\leq\lambda$, and it contains $m_\lambda$.
\end{proof}
Now we can prove the main theorem of this section.

\begin{theorem}
$\{X'_{-\lambda}(\alpha^{-1})|\lambda \in \mathcal {P}\}$ is a basis
for $\Lambda_F$.
\end{theorem}
\begin{proof} We only need to show that for each positive integer
$n$, the set $\{X'_{-\lambda}(\alpha^{-1})|\lambda\vdash n\}$ is
linearly independent. Suppose it is linearly dependent, there are
relatively prime polynomials $f_\lambda(\alpha)\in
\mathbb{Q}[\alpha],\lambda\vdash n$ such that $\sum_{\lambda\vdash
n}f_\lambda(\alpha)X'_{-\lambda}(\alpha^{-1})=0$. Setting $\al=1$ it
follows that $\sum_{\lambda\vdash n}f_\lambda(1)X'_{-\lambda}(1)=0$,
or $\sum_{\lambda\vdash
n}f_\lambda(1)b_{\lambda}(1)^{-1}X_{-\lambda}(1)=0$. Note that the
coefficients $f_{\lambda}(1)$ are not all zero since $\alpha-1$ is
not the common devisor of the $f_\lambda(\alpha)$'s. This
contradicts to the fact that the $X_{-\lambda}(1)$'s are independent.
\end{proof}

For convenience we list the properties of this basis (generated
by vertex operators) in the following.
\begin{proposition}\label{makecase}
The basis $X'_{-\lambda}(\alpha^{-1})$'s has the following
properties.

(a) $X'_{-\la}(\alpha^{-1})=Q_{\la}(\alpha^{-1})$, for $\la=((k+1)^s)\in\mathcal{P}$;
%\label{markedcase}

(b) $X'_{-\la}(\alpha^{-1})=\frac{-s}{2(\alpha^{-2}+s)}Q_{\la}(\alpha^{-1})$, for $\la=((k+1)^s,k)\in\mathcal{P}$;

(c) $X'_{-\lambda}(\alpha^{-1})$ is a linear combination of the
$q_\mu(\alpha^{-1})'s$ with $l(\mu)\leq l(\lambda)$,

(d) $X'_{-\lambda}(\alpha^{-1})$ is orthogonal to
 $q_\mu(\alpha^{-1})$ if $l(\mu')>l(\lambda')$.
\end{proposition}
\begin{proof}
Equality (a) comes from Corollary \ref{C:explicitcoefficient}.
Equality (b) comes from $t=1$ case of Theorem \ref{L:squareJack} with the coefficient fixed
by using $i=1$ case of formula (\ref{colambdai}) in Proposition \ref{P:alambdai}.
Property (c) is from formula (\ref{F:Xlambda}) in the proof of Corollary \ref{L:squareJack}.
Property (d) comes from Proposition \ref{P:Xlambdaqmu}.
\end{proof}
\begin{remark} The vertex operator $X_{\lambda}$ and $X'_{\lambda}$ differ by
the Dyson's constant term in $\prod_{1\leq i\neq j\leq
l(\lambda)}(1-x_ix_j^{-1})^{\alpha}$ up to a sign.
\end{remark}

To end this section, we give a new vertex operator realization of
Schur functions of rectangular shapes with part of the last column
removed (cf. \cite{J1}). Note that $l(\mu)\leq\l(\la)$ is a necessary condition for
$q_\mu$ to appear in the expansion of $X_{-\la}(\al^{-1})$. As a
byproduct of Lemma \ref{L:H1analysis}, we have
\begin{corollary}
For partition $\la=((k+1)^s,k^t)$ of length $s+t$,
$X_{-\lambda}(1)=c_\la S_\lambda$ for some non-zero constant
$c_\la$.
\end{corollary}

\section{Stanley's positivity conjecture and Filtration of rectangles}
From now on, the parameter for Jack functions is assumed to be
$\alpha$, subsequently the scalar
product is the original one, i.e. $\langle p_\la,p_\mu\rangle=\delta_{\la\mu}z_\la\al^{l(\la)}$. The following conjecture
about Littlewood-Richardson coefficients of the Jack function was made by
Stanley \cite{S}, referred to as Stanley's LR conjecture. Denote $\mathbb{Z}_{\geq0}[\al]$ as the set of polynomials in $\al$ with non-negative
integer coefficients.
\begin{conjecture}\label{C:Stanley's conjecture} The LR-cofficient
$C_{\mu\nu}^\la(\al):=$ $\langle J_\mu
J_\nu,J_\la\rangle$ is in
$\mathbb{Z}_{\geq0}[\al]$ for any triple of partitions $(\mu,\nu,\la)$.
\end{conjecture}

\begin{remark}\label{R:Stanley's conjecture}
Stanley's Pieri formula for Jack function implies his conjecture in the case of $l(\nu)=1$.
%\begin{theorem} \cite{S} \label{T:Stanley}
%For partitions $\mu$, $\lambda$, and positive integer $n$, the
%expansion coefficient $\langle J_nJ_\mu,J_\lambda\rangle\neq 0$ if
%and only if $\mu\subset\lambda$ and $\lambda/\mu$ is a horizontal
%$n$-strip. And in this case we have
% \begin{equation}
%\langle
%J_n(\alpha)J_\mu(\alpha),J_\lambda(\alpha)\rangle=h^\mu_*(\mu_u)h^*_\mu(\mu_b)\cdot
%h^*_n(n)\cdot h^*_\lambda(\lambda_u)h^\lambda_*(\lambda_b).
%\end{equation}
%\end{theorem}
One the other hand, using \cite{KS}, we know that
$C_{\mu\nu}^\la(\al)\in \mathbb Z[\al]$. If the classical
Littlewood-Richardson coefficient $C_{\mu\nu}^\la=1$,  Stanley
further conjectured a form for $C_{\mu\nu}^\la(\al)$ which he verified
under a technical condition, and this contains part of the following Theorem \ref{T:Jack*actsonsquareJack}. Furthermore Stanley showed the latter
conjecture is true for $\nu=(2,1)$ under the condition $C_{\mu\nu}^\la=1$.
We will show that Conjecture \ref{C:Stanley's conjecture} is true for $\nu=(2,1)$.
%(without the condition $C_{\mu\nu}^\la=1)$.
\end{remark}

\subsection{Rectangular and marked rectangular cases}
In the following we prove two special cases of Conjecture \ref{C:Stanley's conjecture}.
\begin{definition}
For a rectangular partition $\la=(r^s)$, and a partition
$\nu=(\nu_1,\nu_2,\cdots,\nu_s)\subset\la$, define the complement partiton
of $\nu$ in $\la$ as
$\overline{\nu}=\la-'\nu=(r-\nu_s,r-\nu_{s-1},\cdots,r-\nu_1)$.
\end{definition}

 \begin{example}\label{E:complement}
 For $\la=(6^5)$, $\nu=(6,5,3,3,0)$, we have $\la-'\nu=(6,3,3,1,0)$. As shown in the following Fig.2 (a), the boxes below the bold line forms an upside-down $\nu$ of $\la$. The partition above the bold line is $\la-'\nu$.
\end{example}

 \vskip 0.0cm
 \setlength{\unitlength}{0.5cm}
\begin{picture}(1,1)

 \put(0,0){\line(1,0){6}}%rectangular
 \put(0,-5){\line(1,0){6}}
 \put(0,0){\line(0,-1){5}}
 \put(6,0){\line(0,-1){5}}
 %horizontals
\put(0,-1){\line(1,0){6}}
\put(0,-2){\line(1,0){6}}\put(0,-3){\line(1,0){6}}\put(0,-4){\line(1,0){6}}
%verticals
\put(1,0){\line(0,-1){5}}\put(2,0){\line(0,-1){5}}\put(3,0){\line(0,-1){5}}
\put(4,0){\line(0,-1){5}}\put(5,0){\line(0,-1){5}}
%Partition lines
\put(0,-4.03){\line(1,0){1.03}}\put(0,-4.06){\line(1,0){1.03}}
\put(1.06,-4.05){\line(0,1){1}}\put(1.03,-4.05){\line(0,1){1}}
\put(1.06,-3.05){\line(1,0){1.97}}\put(1.03,-3.02){\line(1,0){1.98}}
\put(3.06,-3.05){\Large\line(0,1){2}}\put(3.03,-3.05){\Large\line(0,1){2}}
\put(3.05,-1.02){\Large\line(1,0){2.94}}\put(3.05,-1.05){\Large\line(1,0){2.94}}

\put(1.6,-6.2){Fig.2 (a)}
%move 8 unints right
\put(8,0){\line(1,0){6}}%rectangular
 \put(8,-5){\line(1,0){6}}
 \put(8,0){\line(0,-1){5}}
 \put(14,0){\line(0,-1){5}}
 %horizontals
\put(8,-1){\line(1,0){6}}
\put(8,-2){\line(1,0){6}}\put(8,-3){\line(1,0){6}}\put(8,-4){\line(1,0){6}}
%verticals
\put(9,0){\line(0,-1){5}}\put(10,0){\line(0,-1){5}}\put(11,0){\line(0,-1){5}}
\put(12,0){\line(0,-1){5}}\put(13,0){\line(0,-1){5}}
%Partition lines
\put(8,-4.03){\line(1,0){1.03}}\put(8,-4.06){\line(1,0){1.03}}
\put(9.06,-4.05){\line(0,1){1}}\put(9.03,-4.06){\line(0,1){1}}
\put(9.05,-3.05){\line(9,0){1.98}}\put(9.03,-3.02){\line(1,0){1.98}}
\put(11.05,-3.05){\Large\line(0,1){2}}\put(11.02,-3.05){\Large\line(0,1){2}}
\put(11.05,-1.02){\Large\line(1,0){2.94}}\put(11.02,-1.05){\Large\line(1,0){2.94}}

\put(9.15,-1.45){$a$}
\put(9.5,-1.5){\vector(0,-1){1.5}}

\put(9.6,-1.5){\line(1,0){0.1}}
\put(9.8,-1.5){\line(1,0){0.1}}
\put(10,-1.5){\vector(1,0){1}}

\put(11.6,-4.85){$b$} \put(11.5,-4.4){\line(0,1){0.1}}\put(11.5,-4.2){\line(0,1){0.1}}\put(11.5,-4){\vector(0,1){3}}
\put(11.5,-4.5){\vector(-1,0){3.5}}
\put(9.7,-6.2){Fig.2 (b)}
%move 13+3 horizontally right
 \put(16,0){\line(1,0){6}}%rectangular
 \put(16,-5){\line(1,0){6}}
 \put(16,0){\line(0,-1){5}}
 \put(22,0){\line(0,-1){5}}
 %horizontals
\put(16,-1){\line(1,0){6}}
\put(16,-2){\line(1,0){6}}\put(16,-3){\line(1,0){6}}\put(16,-4){\line(1,0){6}}
%verticals
\put(17,0){\line(0,-1){5}}\put(18,0){\line(0,-1){5}}\put(19,0){\line(0,-1){5}}
\put(20,0){\line(0,-1){5}}\put(21,0){\line(0,-1){5}}
%Partition lines
\put(16,-4.03){\line(1,0){1.03}}\put(16,-4.06){\line(1,0){1.03}}
\put(17.05,-4.05){\line(0,1){1}}\put(17.02,-4.05){\line(0,1){1}}
\put(17.05,-3.05){\line(1,0){1.97}}\put(17.02,-3.02){\line(1,0){1.98}}
\put(19.05,-3.05){\Large\line(0,1){2}}\put(19.02,-3.05){\Large\line(0,1){2}}
\put(19.05,-1.02){\Large\line(1,0){2.94}}\put(19.05,-1.05){\Large\line(1,0){2.94}}

\put(17.15,-1.45){$a$} \put(17.5,-1.6){\line(0,-1){0.1}}\put(17.5,-1.8){\line(0,-1){0.1}}\put(17.5,-2){\vector(0,-1){1}}

\put(17.5,-1.5){\vector(1,0){4.5}}
\put(19.1,-4.9){$b$}
\put(19.5,-4.5){\vector(0,1){4.5}}

\put(19.6,-4.5){\line(1,0){0.1}}
\put(19.8,-4.5){\line(1,0){0.1}}
\put(20,-4.5){\vector(1,0){2}}
\put(17.5,-6.2){Fig.2 (c)}
\end{picture}
\vskip 3.4cm
Note that $\nu\subset (r^s)$ is equivalent to that $l(\nu)\leq s$ and $\nu_1\leq r$, and in this case, the Young diagram of $(r^s)-'\nu$ is the set difference of $Y(\la)$ with the upside-down $\nu$ of $\la$. We have the following simple properties for the complement of
partitions with the proof being obvious thus omitted.
\begin{lemma}\label{L:complementofrectangular}
For a rectangular partition $\la=(r^s)$, and partitions
$\nu,\omega\subset\la$, we have
\begin{align*}
&\overline{\overline{v}}=v,\\
&\omega\geq\nu\Rightarrow\overline{\omega}\geq\overline{\nu},\\
&\omega\supset\nu\Rightarrow\overline{\omega}\subset\overline{\nu}.
\end{align*}
\end{lemma}

\begin{lemma}\label{L:qsquareJack}
For $\nu\subset\la=(r^s)$, $\langle J_\omega
q_\nu,J_\lambda\rangle\neq0$ implies $\omega\geq\overline{\nu}$.
\end{lemma}

\begin{proof} $q_\nu^*.J_\la=q_{\nu_s}^*\cdots q_{\nu_1}^*.J_\la$, the
action of $q_{\nu_s}^*\cdots q_{\nu_1}^*$ on $J_\la$ has the effect
of remove $v_1+\cdots+v_i$ squares in the last $i$ row(s) of $\la$.
Here we mean that setting $q_{\nu_s}^*\cdots
q_{\nu_1}^*.q_\la=\sum_\omega C_\omega J_\omega$, if
$C_\omega\neq0$, then $\omega$ is coming from $\la$ by removing
$v_1+\cdots+v_i$ squares in the last $i$ row(s) of $\la$. Thus in
the first $s-i$ row(s), there are at most $v_{i+1}+\cdots+v_s$
squares that has been removed. Thus this kinds of $\omega$ satisfies
$\omega_1+\cdots+\omega_{s-i}\geq
r(s-i)-(v_{i+1}+\cdots+v_s)=(r-v_s)+\cdots+(r-v_{i+1})$, which means
$\omega\geq\overline{\nu}$.
\end{proof}

The following proposition claims the correctness of Stanley's L-R
conjecture in the case that $\la$ is of rectangular shape.
\begin{theorem}\label{T:Jack*actsonsquareJack}
For rectangular partition $\la=(r^s)$, $\langle J_\mu
J_\nu,J_\la\rangle\neq0$ if and only if $\nu\subset\la$ and
$\mu=\overline{\nu}$. Furthermore, in this case we have
$\langle J_\mu J_\nu,J_\la\rangle=h_*^\mu(\mu)h^*_\nu(\nu)h_1^\la(\mu)h_2^\la(\nu),$ where
\begin{align*}
&h_1^\la(\mu)=\prod_{(i,j)\in\mu}[(\mu'_j-i)+\al(r-j+1)],\\
&h_2^\la(\nu)=\prod_{(i,j)\in\nu}[(s-i+1)+\al(j-1)].
\end{align*}
\end{theorem}

\begin{remark}
Note that except the familiar terms $h_*^\mu(\mu)$ and
$h^*_\nu(\nu)$, there are two \emph{new} terms $h_1^\la(\mu), h_2^\la(\nu)$ in $C_{\mu\nu}^\la(\al)$. For instance, let us see $\la=(6^5), \nu=(6,5,3,3)$ as in Example \ref{E:complement}. The partition $\mu=\la-'\nu$ lies above the bold line in Fig.2 (a). Fig.2 (b) shows the familiar terms, the box $a$ in $\mu$ contributes $2+\al$, and box $b$  in
$(6^5)-\mu$ contributes $3+4\al$. Fig.2 (c) shows the \emph{new} terms, where the box $a$ contributes $1+5\al$ and the box $b$ contributes $5+2\al$. So, in (b) or (c), we associate a linear polynomial of
$\al$ to each box, the number of the boxes where the solid line passes
horizontally (resp. vertically) is the coefficient of $\al$
( resp. the constant).
\end{remark}

\begin{proof} Set $J_\nu=\sum_{\omega\geq\nu}c_\omega q_\omega$, $\langle
J_\mu J_\nu,J_\la\rangle\neq0$ implies at least one of $\langle J_\mu
q_\omega,J_\la\rangle\neq0$ for some $\omega\geq\nu$, thus
$\mu\geq\overline{\omega}\geq\overline{\nu}$ by Lemma
\ref{L:qsquareJack} and Lemma \ref{L:complementofrectangular}.
Similarly, we have $\nu\geq\overline{\mu}$, which leads to
$\overline{\nu}\geq\mu$ by Lemma \ref{L:complementofrectangular}.
Now $\mu=\overline{\nu}$ is set up.

Now we can set $J_\nu^*.J_\la=C_{\la,\nu}J_{\overline{\nu}}$, and then $\langle J_\mu J_\nu,J_\la\rangle=\langle
J_{\overline{\nu}},J_\nu^*. J_\la\rangle=C_{\la,\nu}\langle
J_{\overline{\nu}},J_{\overline{\nu}}\rangle=C_{\la,\nu}h_{\overline{\nu}}^*(\overline{\nu})h_*^{\overline{\nu}}(\overline{\nu})$. We need to
determine the coefficient $C_{\la,\nu}$. We can first write
$J_\nu^*=\sum_{\omega\geq\nu}c_\omega q_\omega^*$ with $c_\nu=h^*_\nu(\nu)$ by Lemma \ref{L:triangular}, then the only
term that counts is $c_\nu q_\nu^*$ by Lemma \ref{L:qsquareJack}. The action of
$q_{\nu_1}^*$ on $J_{\la}$ removes $\nu_1$ squares in the last row
of $\la$, resulted in, say $C_{\la,\nu_1}J_{\overline{(\nu_1)}}$,
where $C_{\la,\nu_1}$ is explicitly known by Stanley's Pieri formula
for Jack function. Then we apply the action $q_{\nu_2}^*$ on
$J_{\overline{(\nu_1)}}$. There maybe more than one ways to remove
$\nu_2$ squares horizontally from $\overline{(\nu_1)}$, but the only
way that contributes to $J_{\overline{\nu}}$ is to remove $\nu_2$
squares in the second last row, which results in, say
$C_{\la,(\nu_1,\nu_2)}J_{\overline{(\nu_1,\nu_2)}}$, with
$C_{\la,(\nu_1,\nu_2)}$, which is known by Stanley's formula. Continuing
this process, we find that $C_{\la,\nu}=h^*_\nu(\nu)\cdot(C_{\la, \nu_1}C_{\la, (\nu_1,\nu_2)}\cdots)=h^*_\nu(\nu)\cdot(h_1^\la(\mu)h_2^\la(\nu)/h_{\overline{\nu}}^*(\overline{\nu}))$.
\end{proof}

\begin{remark}
1. It is interesting that the expression is symmetric in
$\mu$ and $\nu$, but the symmetry is not at all clear
from this form of the formula.\\
2. By the same method, we can easily generalize this theorem to
Macdonald case using the Pieri formula.
\end{remark}

\begin{corollary}\label{C:StanleyL-R}
For a marked rectangular partition $\la=(r^{s-1},r-n)$, the matrix coefficient
$\langle
J_\mu J_\nu,J_\la\rangle\neq0$ if and only if
$((r^s)-'\nu)-\mu$ is a horizontal $n$-strip. In this
case $\langle J_\mu J_\nu,J_\la\rangle$ is a product of nonzero
linear polynomials in $\mathbb{Z}_{\geq0}[\al]$, Explicitly, denoting $\overline{\nu}=(r^s)-'\nu$, we have
\begin{align}
\langle J_\mu
J_\nu,J_\la\rangle=&\frac{(r-n)!}{r(r-1)\cdots(n+1)\prod_{i=0}^{n-1}[s+i\al]}\cdot\\\nonumber
&\frac{h_*(\overline{\nu}_u)h^*(\overline{\nu}_b)}{h_*(\mu_u)h^*(\mu_b)}n!\al^n\langle
J_{\overline{\nu}}J_\nu,J_{r^s}\rangle,
\end{align}
where  % and
%$\mu_b$ (resp. $\overline{\nu}_b$) is the set of the based squares of $\mu$ (resp. $\overline{\nu}$) for the pair $\mu\subset\overline{\nu}$.
$\mu_b$ and $\overline{\nu}_b$ ( resp. $\mu_u$ and $\overline{\nu}_u$) are the sets of based (resp. un-based) boxes of $\mu$ and $\overline{\nu}$ respectively for the pair $\mu\subset\overline{\nu}$.
% in the process
%of adding $n$ squares horizontally to $\mu$ to reach
%$\overline{\nu}$, and similarly for $\mu_u$, $\overline{\nu}_u$ and
%$\overline{\nu}_b$.
\end{corollary}
\begin{proof}
First we have
\begin{align*}
J_n^*.J_{r^s}&=\frac{r\al\cdots(n+1)\al}{(r-n)\al\cdots1\al}[s+(n-1)\al]\cdots[s+0\al]J_\la,\\
J_\la&=\frac{(r-n)!}{r\cdots
(n+1)\prod_{i=0}^{n-1}(s+i\al)}J_n^*.J_{r^s},\\
 \langle J_\mu
J_\nu,J_\la\rangle&=\frac{(r-n)!}{r(r-1)\cdots(n+1)\prod_{i=0}^{n-1}[s+i\al]}\langle
J_nJ_{\mu},J_\nu^* J_{r^s}\rangle.
\end{align*}
By Theorem \ref{T:Jack*actsonsquareJack}, it does not vanish if and
only if $((r^s)-'\nu)-\mu$ is a horizontal $n$-strip. Then
the coefficient of $J_{\overline{\nu}}$ in $J_nJ_\mu$ is
$\frac{h_*(\overline{\nu}_u)h^*(\overline{\nu}_b)}{h_*(\mu_u)h^*(\mu_b)}n!\al^n$
by Pieri formula of Jack functions \cite{S}. Combining this with
theorem \ref{T:Jack*actsonsquareJack} gives the corollary.
\end{proof}

\begin{remark} This proposition implies that Conjecture \ref{C:Stanley's conjecture} is true if $\la$ is a rectangle with part of the last row removed. We can also use this method to treat the case of rectangular partitions with part of the last column removed, and prove that Conjecture \ref{C:Stanley's conjecture} is true if $\la=((k+1)^sk^t)$ with $k, s, t\in\mathbb{Z}_{\geq0}$.
\end{remark}
We can generalize Theorem \ref{T:Jack*actsonsquareJack} to the
following situation, which is also useful when we consider the
$\nu=(2,1)$ case of Stanley's conjecture.

\begin{proposition}\label{P:Jstar remove in one corner}
Assume that $\langle J_\nu^*.J_\la,J_\mu\rangle\neq 0$, and $\la-\mu$ is
an upside-down $\omega$ of a corner of $\la$ with $|\omega|=|\nu|$, then $\omega=\nu$ and $\langle
J_\la,J_\mu J_\nu\rangle$ is the product of linear polynomials of
$\al$ with non-negative integer coefficients.
\end{proposition}

\begin{proof} As the proof of Theorem \ref{T:Jack*actsonsquareJack}, we
have $\omega\geq\nu$. Note that $\langle J_\la,J_\mu
J_\nu\rangle=\langle J_\nu^*.J_\la,J_\mu\rangle\neq 0$ is equivalent
to $\langle J_{\la'},J_{\mu'} J_{\nu'}\rangle\neq 0$, and that $\la'-\mu'$ is
an upside-down $\omega'$ of a corner of $\la'$ with $|\omega'|=|\nu'|$, thus we also
have $\omega'\geq\nu'$. Subsequently $\omega=\nu$.  The proof for
the other statement is the same as that of Theorem
\ref{T:Jack*actsonsquareJack}.
\end{proof}

\subsection{$v=(2,1)$ case }
 The following
proposition confirms the Stanley's conjecture for the case of
$\nu=(n,1)$ subject to a simpler technical condition, but without the condition
$C_{\mu\nu}^\la=1$.
\begin{proposition}\label{P:stanley1}
 For $\nu=(n,1)$, $\langle J_\mu J_\nu, J_\la\rangle$ is a product of linear polynomials of
 $\al$ with nonnegative integer coefficients
  if the boxes
of $\la-\mu$ lie on two rows of $\la$. In particular, for $\nu=(2,1)$, $\langle
J_\mu J_\nu, J_\la\rangle\in \mathbb{Z}_{\geq0}[\al]$.
\end{proposition}

\begin{proof}
We first prove the $(n,1)$-case. Assume that among the $n+1$ boxes of $\la-\mu$, there are $n-a$ boxes in
the same row of one corner, and $a+1$ boxes in the same row of a
lower corner. We mark the two right-most boxes on the two rows by $A$ and $B$. We also denote $t$ (resp. $s$)
as the number of rows (resp. columns) between the rows (resp.
columns) with removed boxes. See the following Fig.3 (a).

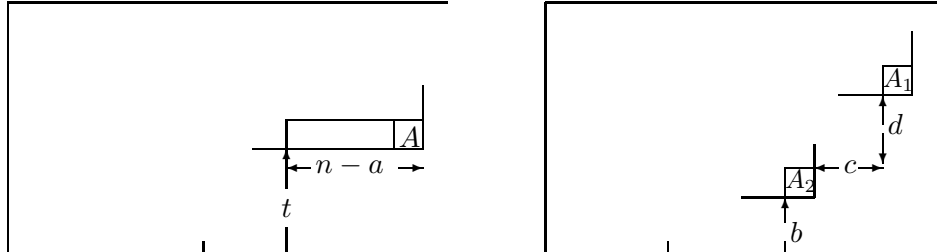
 \setlength{\unitlength}{0.65cm}
\begin{picture}(1,1)
 \put(0,0){\line(1,0){9}}%rectangular
 \put(0,0){\line(0,-1){6.5}}%
 %first corner
 \put(8.5,-1.7){\line(0,-1){1.3}}
 \put(8.5,-3){\line(-1,0){3.5}}

%second corner
  \put(4,-4.9){\line(0,-1){1.1}}
 \put(4,-6){\line(-1,0){3.2}}

 %removed (1,1) 2 different corner
 \put(8.5,-2.4){\line(-1,0){0.6}}
  \put(7.9,-2.4){\line(0,-1){0.6}}
 \put(8.5,-2.4){\line(-1,0){2.8}}
\put(5.7,-2.4){\line(0,-1){0.6}}
\put(8,-2.95){${A}$}
 \put(4,-5.4){\line(-1,0){0.6}}
  \put(3.4,-5.4){\line(0,-1){0.6}}
 \put(4,-5.4){\line(-1,0){2.5}}
\put(1.5,-5.4){\line(0,-1){0.6}}
\put(3.5,-5.9){$B$}
%mark a+1
 \put(2,-6.3){\vector(-1,0){0.5}}
  \put(2.1,-6.5){$a+1$}
  \put(3.5,-6.3){\vector(1,0){0.5}}
%mark n-a
 \put(6.2,-3.4){\vector(-1,0){0.5}}
  \put(6.3,-3.5){$n-a$}
  \put(8,-3.4){\vector(1,0){0.5}}
%mark s
 \put(4.6,-5.4){\vector(-1,0){0.6}}
  \put(4.7,-5.6){$s$}
  \put(5.1,-5.4){\vector(1,0){0.6}}
%mark t
    \put(5.7,-3.8){\vector(0,1){0.8}}
  \put(5.6,-4.4){$t$}
  \put(5.7,-4.6){\vector(0,-1){0.8}}
%-----------\move 11 units right
%\put(3.5,-7.5){Fig.3 $\nu=(2,1)$, case (4)}
 \put(11,0){\line(1,0){8}}%rectangular
 \put(11,0){\line(0,-1){6.5}}%

 \put(18.5,-0.6){\line(0,-1){1.3}}
 \put(18.5,-1.9){\line(-1,0){1.5}}

%second corner
  \put(16.5,-2.9){\line(0,-1){1.1}}
 \put(16.5,-4){\line(-1,0){1.5}}
%third corner
   \put(13.5,-4.9){\line(0,-1){1.1}}
 \put(13.5,-6){\line(-1,0){1.5}}
 %removed (1,1,1) in three different corner
 \put(18.5,-1.3){\line(-1,0){0.6}}
  \put(17.9,-1.3){\line(0,-1){0.6}}
\put(17.90, -1.74){{\small $A_1$}}
 \put(16.5,-3.4){\line(-1,0){0.6}}
  \put(15.9,-3.4){\line(0,-1){0.6}}
\put(15.90, -3.8){{\small $A_2$}}
 \put(13.5,-5.4){\line(-1,0){0.6}}
  \put(12.9,-5.4){\line(0,-1){0.6}}
\put(12.90, -5.85){{\small $A_3$}}
%mark the distances a b
  \put(14.5,-5.4){\vector(-1,0){1}}
  \put(14.6,-5.5){$a$}
  \put(14.9,-5.4){\vector(1,0){1}}

  \put(15.9,-4.5){\vector(0,1){0.5}}
  \put(16.0,-4.9){$b$}
  \put(15.9,-4.9){\vector(0,-1){0.5}}
%mark the distances c d
 \put(17,-3.4){\vector(-1,0){0.5}}
  \put(17.1,-3.5){$c$}
  \put(17.4,-3.4){\vector(1,0){0.5}}

   \put(17.9,-2.5){\vector(0,1){0.6}}
  \put(18,-2.7){$d$}
  \put(17.9,-2.7){\vector(0,-1){0.6}}
  \put(3.4,-7.6){Fig.3 (a)}
  \put(14.4,-7.6){Fig.3 (b)}
\end{picture}
\vskip 5.2cm
  We can assume that $n-a,a+1\geq1$. In fact, if one of them is zero, then $\la-\mu$ lies on one row of $\la$, and $\la-\mu$ should be an upside-down $(n+1)$ of a corner of $\la$. This is a special case of the Proposition \ref{P:Jstar remove in one corner}.
Note that we have $J_{(n,1)}=(n-1)!\al^n[(1+n\al)q_nq_1-(n+1)q_{n+1}]$, and $\langle J_\nu J_\nu, J_\la\rangle=\langle J_\mu, J_\nu^*.J_\la\rangle$. Using Pieri formula, we compute the coefficient $c$ of $J_\mu$ in $q_{n+1}^*.J_\la$. For $q_n^*q_1^*.J_\la$, we compute coefficient $a_1$ of $J_{\la-A}$ in $q_1^*.J_\la$, coefficient $a_2$ of $J_\mu$ in $q_{n}^*.J_{\la-A}$. We also compute coefficient $b_1$ of $J_{\la-B}$ in $q_1^*.J_\la$, coefficient $b_2$ of $J_\mu$ in $q_{n}^*.J_{\la-B}$. Then $\langle J_\mu J_{(n,1)},J_\la\rangle=(n-1)!\al^n[(1+n\al)(a_1a_2+b_1b_2-(n+1)c)]\langle J_\mu, J_\mu\rangle$ is found to be a product of linear polynomials of $\al$ with nonnegative integer coefficients.
%%%The following computation is re-thought and should be absolutely right, but it is a little bit too bulky to write it here.
%\begin{align*}
%&\langle J_\mu J_{(n,1)},J_\la\rangle=(n-1)!\al^n\frac{h^*_\la(R')}{h^*_\mu(R')}
%\frac{h_*^\la(C')}{h_*^\mu(C')}h^*_\mu(\mu)h_*^\mu(\mu)\\
%&\qquad \qquad\qquad \qquad \cdot [(1+n\al)(j_A+j_{B})-(n+1)j_{n+1}],\\
%\end{align*}
%where $R'=R-R\cap C$, $C'=C-R\cap C$, and $R$ (resp. $C$) is the set of the boxes of $\mu$ on a row (resp. column) of $\la$ which intersects $\la-\mu$.  The terms $j_{n+1},j_A,j_B$, multiplied by
%$h_*^\mu(C\cap R)(t+(s+1)\al)$, and with $\prod_{1\leq j\leq
%a}[(t+2+(s+n-j)\al)(1+(j-1)\al)]\prod_{0\leq j\leq n-a-2}(1+j\al)$
%factored out, becomes respectively,
%
%$j'_{n+1}=(1+a\al)(t+2+(s+n)\al)(1+(n-a-1)\al)(t+(s+1)\al),$\\
%$j'_A=(n-a)\frac{t+1+(s+n+1)\al}{t+1+(s+n-a)\al}(t+2+(s+n-a-1)\al)(1+a\al)(t+(s+1)\al),$\\
%$j'_B=(a+1)\frac{t+(s+n-a+1)\al}{t+1+(s+n-a)\al}(t+2+(s+n)\al)(t+1+s\al)(1+(n-a-1)\al)).$\\
%It can be computed that:
%$(1+n\al)(j'_A+j'_B)-(n+1)j'_3\\
%=(2+(n-1)\al)(t+1+s\al)(t+1+(s+n+1)\al)(a+1)(n-a)\al$.\\

Now we prove the case of $(2,1)$. The idea is same as the $(n,1)$-case.  Among the three boxes of $\la-\mu$,
 if any two of them are in the same column of $\la$, then by Pieri formula only $q_2^*q_1^*$ contributes to $J_\mu$ in $J_{(2,1)}^*.J_\la= \al^2[(1+2\al)q_2^*q_1^*-3q_3^*].J_\la$, and it can easily be seen that the proposition is true. If any two of them are in a same column of $\la$, it is also true by the $(n,1)$-case. Thus we only need to prove the case when the three boxes of $\la-\mu$ are in three different corners of $\la$, see Fig.3 (b). We mark the three boxes of $\la-\mu$ by $A_1, A_2, A_3$, $d$ (resp. $c$) is the number of rows (resp. columns) between boxes $A_1$ and $A_2$. Similarly we define $a, b$ as shown in Fig.3 (b). A row (resp. column) of $\mu$ is called a critical row (resp. column) if it lies on a row (resp. column) of $\la$ which intersects $\la-\mu$. Now
 \begin{align*}
&\langle J_\mu J_{(2,1)},J_\la\rangle=\al^2h^*_\la(R-R\cap C)h_*^\la(C-R\cap C)h^*_\mu(\mu-R)h_*^\mu(\mu-C)f(\al),
\end{align*}
where $R$ (resp. $C$) is the union of the critical rows (resp. columns) of  $\mu$ , and $f(\al)=\sum_{i=0}^7c_{i}\al^i,$
with $c_i\in\mathbb{Z}_{\geq0}[a,b,c,d]$, thus the positivity follows.
\end{proof}

\subsection{Filtration of rectangular partitions} Using Theorem
\ref{T:Jack*actsonsquareJack}, we can derive Jack function of any
shape starting from Jack functions of rectangular shapes, which were
already realized in \cite{CJ} using vertex operator technique.
\begin{definition}
Define maps $\mathfrak{R}, \mathfrak{C}$ on $\mathcal{P}$ by $\mathfrak{R}(\la)=(l(\la')^{l(\la)})$, $\mathfrak{C}(\la)=\mathfrak{R}(\la)-'\la$. Define $R_i(\la)=\mathfrak{R}(\mathfrak{C}^{i-1}(\la))$ ($i\geq1$, $\mathfrak{C}^0(\la)=\la$). Set $s=n_c(\la)$ the number of corner of $\la$, $R_\la=(R_1,\cdots,R_s)$ is called the rectangular filtration of $\la$.
\end{definition}
\begin{example}\label{E:rectangularfiltration}
 For $\la=(6,3,3,2,1)$, note that $l(\la')=\la_1$, so $R_1(\la)=\mathfrak{R}(\la)=(6^5)$. $\mathfrak{C}^1(\la)=\mathfrak{C}(\la)=(6^5)-'(6,3,3,2,1)=(5,4,3,3)$, $R_2(\la)=\mathfrak{R}(\mathfrak{C}^1(\la))=(5^4)$. Similarly, it can be found that $\mathfrak{C}^2(\la)=\mathfrak{C}(\mathfrak{C}(\la))=(2,2,1)$. $R_3(\la)=(2^3)$, $\mathfrak{C}^3(\la)=(1)$ and $R_4(\la)=(1^1)$. Thus the filtration of $\la$ is $R_\la=((6^5), (5^4), (2^3), (1^1))$. The following figure illustrates the pictorial construction of the rectangular filtration. The $\la$ is the partition above the bold line in picture (1), the whole
rectangle in picture (1) (with diagonal $a_0a_1$) is $R_1$. $R_1-\la$ is the subset under the bold line in picture (1), it is an upside-down $\mathfrak{C}^1(\la)=(5,4,3,3)$ of (the corner of) $R_1$. The whole rectangle (2) (with diagonal $a_1a_2$) is $R_2$, it is also the sub-rectangle with diagonal $a_1a_2$ in picture (1). Similarly, the whole rectangle in picture (i) with diagonal $a_{i-1}a_i$ is $R_i$. In fact, all rectangle $R_i$ is marked in picture (1) with diagonal $a_{i-1}a_i$.
\end{example}
 \setlength{\unitlength}{0.5cm}
\begin{picture}(1,1)
\put(-0.1,-0.2){$\bullet$} \put(-0.25,0.2){$a_0$}
 \put(0,0){\line(1,0){6}}%rectangular
 \put(0,-5){\line(1,0){6}} \put(0,0){\line(0,-1){5}} \put(6,0){\line(0,-1){5}}
\put(6.05,-5.2){$\bullet$}\put(5.9,-4.8){$a_1$}
\put(0.9,-1.2){$\bullet$}\put(0.75,-0.8){$a_2$}
  \put(2.9,-4.2){$\bullet$}
 \put(2.75,-3.8){$a_3$}
   \put(1.9,-3.2){$\bullet$}
 \put(1.75,-2.8){$a_4$}
 %horizontals
\put(0,-1){\line(1,0){6}}
\put(0,-2){\line(1,0){6}}\put(0,-3){\line(1,0){6}}\put(0,-4){\line(1,0){6}}
%verticals
\put(1,0){\line(0,-1){5}}\put(2,0){\line(0,-1){5}}\put(3,0){\line(0,-1){5}}
\put(4,0){\line(0,-1){5}}\put(5,0){\line(0,-1){5}}
%Partition lines
\put(0,-5.03){\line(1,0){1.03}}\put(0,-5.06){\line(1,0){1.03}}%line1
\put(1.05,-5.05){\line(0,1){1}}\put(1.03,-5.05){\line(0,1){1}}%line2
\put(1.02,-4.06){\line(1,0){0.96}}\put(1.02,-4.03){\line(1,0){0.96}}%line3
\put(2.04,-4.05){\line(0,1){0.96}}\put(2.02,-4.02){\line(0,1){0.96}}%line4
\put(2.05,-3.05){\line(1,0){0.96}}\put(2.02,-3.02){\line(1,0){0.96}}%line5
\put(3.05,-3.05){\Large\line(0,1){2}}
\put(3.05,-1.02){\Large\line(1,0){2.94}}\put(3.05,-1.05){\Large\line(1,0){2.94}}
%%second rectangular move 7 unit right
 \put(8,-5){\line(1,0){5}}
 \put(13,-1){\line(0,-1){4}}
 \put(13.05,-5.2){$\bullet$}
 \put(12.9,-4.8){$a_1$}
 \put(7.9,-1.2){$\bullet$}
 \put(7.75,-0.8){$a_2$}
  \put(9.9,-4.2){$\bullet$}
 \put(9.75,-3.8){$a_3$}
   \put(8.9,-3.2){$\bullet$}
 \put(8.75,-2.8){$a_4$}
 %horizontals
\put(8,-1){\line(1,0){5}}
\put(8,-2){\line(1,0){5}}\put(8,-3){\line(1,0){5}}\put(8,-4){\line(1,0){5}}
%verticals
\put(8,-1){\line(0,-1){4}}  \put(9,-1){\line(0,-1){4}} \put(10,-1){\line(0,-1){4}}
\put(11,-1){\line(0,-1){4}} \put(12,-1){\line(0,-1){4}}
%Partition lines
\put(8.05,-5.05){\line(0,1){1}}\put(8.03,-5.05){\line(0,1){1}}%line2
\put(8.02,-4.06){\line(1,0){0.96}}\put(8.02,-4.03){\line(1,0){0.96}}%line3
\put(9.04,-4.05){\line(0,1){0.96}}\put(9.02,-4.02){\line(0,1){0.96}}%line4
\put(9.05,-3.05){\line(1,0){0.96}}\put(9.02,-3.02){\line(1,0){0.96}}%line5
\put(10.02,-3.02){\Large\line(0,1){2}}\put(10.03,-3.02){\Large\line(0,1){2}}
\put(10.05,-1.02){\Large\line(1,0){2.94}}\put(10.05,-1.05){\Large\line(1,0){2.94}}
%%third rectangule move 7 units right
 \put(14.9,-1.2){$\bullet$}\put(14.75,-0.8){$a_2$}  \put(16.9,-4.2){$\bullet$}\put(16.75,-3.8){$a_3$}
   \put(15.9,-3.2){$\bullet$}\put(15.75,-2.8){$a_4$}
 %horizontals
\put(15,-1){\line(1,0){2}}
\put(15,-2){\line(1,0){2}}\put(15,-3){\line(1,0){2}}\put(15,-4){\line(1,0){2}}
%verticals
\put(15,-1){\line(0,-1){3}}  \put(16,-1){\line(0,-1){3}} \put(17,-1){\line(0,-1){3}}
%Partition lines
\put(15.02,-4.06){\line(1,0){0.96}}\put(15.02,-4.03){\line(1,0){0.96}}%line3
\put(16.04,-4.05){\line(0,1){0.96}}\put(16.02,-4.02){\line(0,1){0.96}}%line4
\put(16.05,-3.05){\line(1,0){0.96}}\put(16.02,-3.02){\line(1,0){0.96}}%line5
\put(17.02,-3.02){\Large\line(0,1){2}}\put(17.03,-3.02){\Large\line(0,1){2}}%line6
%%forth rectangule move 5 units right
  \put(21.9,-4.2){$\bullet$}
 \put(21.75,-3.8){$a_3$}
   \put(20.9,-3.2){$\bullet$}
 \put(20.75,-2.8){$a_4$}
 %horizontals
\put(21,-3){\line(1,0){1}}
\put(21,-4){\line(1,0){1}}
%verticals
 \put(21,-3){\line(0,-1){1}} \put(22,-3){\line(0,-1){1}}
%Partition lines
\put(21.04,-4.05){\line(0,1){0.96}}\put(21.02,-4.02){\line(0,1){0.96}}%line4
\put(21.05,-3.05){\line(1,0){0.96}}\put(21.02,-3.02){\line(1,0){0.96}}%line5
\put(6,-7){Fig.3 rectangular filtration}
\put(2.8,-6){(1)}\put(10.2,-6){(2)}\put(15.7,-6){(3)}\put(21,-6){(4)}
 \end{picture}
 \vskip 3.5cm

\begin{lemma}\label{L:RC}
 For a nonzero partition $\la$, $\mathfrak{R}(\la)$ is the minimum rectangular partition containing $\la$ with respect to the partial order $\subset$, and $n_c(\mathfrak{C}(\la))=n_c(\la)-1$.
\end{lemma}
\begin{proof}
If a rectangular partition $\nu=(a^b)$ contains $\la$, then we should have $a\geq \la_1$ and $b\geq l(\la)$ which means $\mathfrak{R}(\la)\subset\nu$. This proves the minimum of $\mathfrak{R}(\la)$. To prove the statement about the number of corners. We write $\la=(a_1^{n_1}\cdots a_s^{n_s})$, with $a_1>a_2>\cdots>a_s>0$ and $m_{a_i}(\la)=n_i>0$. By definition of the compliment partiton we have
\begin{align}
\mathfrak{C}(\la)&=(a_1^{n_1+\cdots+n_s})-'(a_1^{n_1}a_2^{n_2}\cdots a_s^{n_s})\\\nonumber
&=((a_1-a_s)^{n_s}(a_1-a_{s-1})^{n_{s-1}}\cdots (a_1-a_2)^{n_{2}}).
\end{align}
Thus $n_c(\la^1)=s-1$ and the lemma follows because $n_c(\la)=s$.
\end{proof}
We need the following lemma in the proof the realization of Jack functions of general shapes.
\begin{lemma}\label{L:filtrationiterative} If $R_\la=(R_1,\cdots,R_s)$ is the rectangular filtration of $\la$, $s\geq 2$, then $(R_2, R_3,\cdots,R_s)$ is a rectangular filtration of $\nu=\mathfrak{C}(\la)$.
\end{lemma}
\begin{proof}
By Lemma \ref{L:RC}, $n_c(\nu)=s-1$. Thus we can assume that $R_\nu=(R^\nu_1,\cdots,R^\nu_{s-1})$ is the rectangular filtration of $\nu$. By definition, we have $R^\nu_i=\mathfrak{R}(\mathfrak{C}^{i-1}(\nu))=\mathfrak{R}(\mathfrak{C}^{i-1}(\mathfrak{C}(\la)))=\mathfrak{R}(\mathfrak{C}^i(\la))=R_{i+1}$.
\end{proof}
We would like to give a direct description of the rectangular filtration of a partition. For $\overrightarrow{n}=(n_1,\cdots, n_s)\in\mathbb{Z}^s$, define $\overrightarrow{n}^*=(n^*_1,\cdots, n^*_s)\in\mathbb{Z}^s$ by setting $n^*_{2i+1}=\sum_{i<j< s-i+1} n_j$, $n^*_{2i}=\sum_{i<j\leq s-i+1} n_j$.
For a partition $\la=(a_1^{n_1}a_2^{n_2}\cdots a_s^{n_s})$, where
$m_{a_i}(\la)=n_i>0$, we call the vector $\overrightarrow{n}=(n_1,n_2,\cdots,n_s)$ the
multiplicity type of $\la$.

From the Young diagram of $\la$ we can easily prove the following lemma.
\begin{lemma}For a partition $\la=(a_1^{n_1}a_2^{n_2}\cdots a_s^{n_s})$ with
$m_{a_i}(\la)=n_i>0$, the multiplicity type of $\la'$ is $\overrightarrow{p}=(p_1,p_2,\cdots,p_s)$ where
$p_i=a_{s+1-i}-a_{s+2-i}$ with $a_{s+1}=0$.
\end{lemma}
\begin{lemma}\label{L:starinductive}
For $\overrightarrow{n}=(n_1,\cdots, n_s)\in\mathbb{Z}^s$ ($s\geq2$), set $\overrightarrow{m}=(n_s,n_{s-1}\cdots, n_2)$.
 Then we have $\overrightarrow{m}^*=(n^*_2,n^*_3, \cdots, n^*_s)$.
\end{lemma}
\begin{proof}
Set $\overrightarrow{m}=(m_1, m_2\cdots, m_{s-1})$, then $m_j=n_{s+1-j}$ and
$$m^*_{2i}=\sum_{i<j\leq (s-1)-i+1}m_j=\sum_{i<j\leq s-i}n_{s+1-j}=\sum_{i+1\leq k<s+1-i}n_k=n^*_{2i+1}.$$
Similarly it can be computed that $m^*_{2i-1}=n^*_{2i}$.

\end{proof}
\begin{proposition}For a partition $\la$, assume that the multiplicity type $\la$ and $\la'$ are $\overrightarrow{n}=(n_1,\cdots, n_s)$, $\overrightarrow{p}=(p_1,p_2,\cdots,p_s)$, respectively. Then the filtration $R_\la=(R_1,R_2,\cdots,R_s)$ is given by $R_i=({p^*_i}^{n^*_i})$.
%\begin{align*}
%&R_1=(p_1+p_2+\cdots+p_s)^{n_1+n_2+\cdots+n_s},\\
%&R_2=(p_2+\cdots+p_s)^{n_2+\cdots+n_s},\\
%&R_3=(p_2+\cdots+p_{s-1})^{n_2+\cdots+n_{s-1}},\\
%&R_4=(p_3+\cdots+p_{s-1})^{n_3+\cdots+n_{s-1}},\\
%&R_5=(p_3+\cdots+p_{s-2})^{n_3+\cdots+n_{s-2}},\\
% &\cdots,\\
%&R_s=(p_{s'}^{n_{s'}}),
%\end{align*} where $s'=[s/2]+1$, and $[s/2]$ is the
%integer part of $s/2$.
\end{proposition}
\begin{proof}
Set $\la=(a_1^{n_1}a_2^{n_2}\cdots a_s^{n_s})$, with
$m_{a_i}(\la)=n_i>0$.
We make induction on the number of the corners $n_c(\la)=s$. It is obviously true if $n_c(\la)=1$.
 If it is true when $n_c(\la)=s-1$, let us prove the case of $n_c(\la)=s$. Note that $p^*_1=p_1+\cdots+p_s=\sum_{1\leq i\leq s}(a_{s+1-i}-a_{s+2-i})=a_1$ ($a_{s+1}=0$), and $n^*_1=n_1+\cdots +n_s=l(\la)$, we have $R_1=({p^*_1}^{n^*_1})$.

 We set $\nu=\mathfrak{C}(\la)$ then $\nu=((a_1-a_s)^{n_s}(a_1-a_{s-1})^{n_s-1}\cdots (a_1-a_2)^{n_2})$. The multiplicity type of $\nu$ and $\nu'$ are $\overrightarrow{m}=(n_s, n_{s-1},\cdots, n_2)$ and $\overrightarrow{q}=(p_s, p_{s-1},\cdots, p_2)$, respectively. By the inductive assumption and Lemma \ref{L:starinductive} we have $R_\mu=(({q^*_1}^{m^*_1}),\cdots, ({q^*_{s-1}}^{m^*_{s-1}}))=(({p^*_2}^{n^*_2}),\cdots, ({p^*_{s}}^{n^*_{s}}))$. Applying Lemmas \ref{L:filtrationiterative} we have $R_i=({p^*_i}^{n^*_i})$ for $i\geq 2$.
\end{proof}
\
Recall that Jack functions of rectangular shapes were realized by vertex operators in
\cite{CJ}, thus the following theorem gives a vertex operator realization of Jack functions of general shapes.

\begin{theorem}\label{T:iterativeJack}
For a partition $\la$, let $R_\la=(R_1,R_2,\cdots,R_s)$ be the
rectangular filtration of $\la$ and denote the rectangular vertex operator representation
of $R_i$ by $f_i=Q_{R_i}$,
then there is a $c'_\la(\al)\in\mathbb{Q}(\al)\backslash\{0\}$ such that
\begin{align}\label{F:iterativeJack}
Q_\la&=c'_\la(\al)
(\cdots((f_s^*.f_{s-1})^*.f_{s-2})^*.\cdots.f_2)^*.f_1.
\end{align}
\end{theorem}
\begin{proof}
For polynomials $F_1, F_2$ in $\Lambda_F$, we write $F_1\doteq F_2$ if $F_1=cF_2$ for some $c\in F-\{0\}$. We use induction on $s=n_c(\la)$.
If $s=1$, it is trivial. Assuming (\ref{F:iterativeJack}) holds for the case of $s-1$ ($s\geq 2$), then $$Q_{\mathfrak{C}(\la)}\doteq(\cdots((f_s^*.f_{s-1})^*.f_{s-2})^*.\cdots.f_3)^*.f_2.$$
This is due to the fact that $(R_2,\cdots,R_s)$ is the rectangular filtration of $\mathfrak{C}(\la)$ by Lemma \ref{L:filtrationiterative}. Note that by Theorem \ref{T:Jack*actsonsquareJack}, we have $Q_\nu^*.Q_R\doteq Q_{R-'\nu}$ if $R$ is a rectangular partition and partition $\nu\subset R$. Thus we have
\begin{align}
Q_\la&\doteq Q_{\mathfrak{R}(\la)-'\mathfrak{C}(\la)}\doteq Q_{\mathfrak{C}(\la)}^*.Q_{\mathfrak{R}(\la)}\\\nonumber
&\doteq(\cdots((f_s^*.f_{s-1})^*.f_{s-2})^*.\cdots.f_2)^*.f_1.
\end{align}
\end{proof}
\begin{example}Let us look at the $\la$ in Example \ref{E:rectangularfiltration}.  By this theorem, $((Q_1^*Q_{(2^3)})^*Q_{(5^4)})^*Q_{(6^5)}$ is equal to $Q_\la$ up to a non-zero scalar multiplication. We can use our vertex operator formula to find the $Q_{(k^s)}$.

\end{example}
\begin{remark}
1. The coefficient $c'_\la(\al)$ can be explicitly evaluated just as
we have done in the proof of Theorem
\ref{T:Jack*actsonsquareJack}.\\
2. The argument of Theorem  \ref{T:iterativeJack} also works for Macdonald symmetric functions, thus we would have a vertex operator realization for Macdonald functions of general shapes if we had that for rectangular shapes.
\end{remark}
As an application, we have a combinatorial formula for Jack functions of general shapes.
\begin{corollary}\label{C:generalizedFrobenious}
For a partition $\la$, let $R_\la=(R_1,R_2,\cdots,R_s)$ be the
rectangular filtration of $\la$, denote $f_i=Q_{R_i}$,
then we have
\begin{align}\label{F:powerexpansion}
Q_\la=c'_\la(\al)\sum_{\underline{\xi}}\prod_{i=1}^s
g_{R_i,\xi^i}(\al)\prod_{i=1}^{s-1}z_{\dot{\xi}^i}\al^{l(\dot{\xi}^i)}\binom{m(\xi^i)}{m(\dot{\xi}^i)}p_{\dot{\xi}^0}
\end{align}
 where $c'_\la(\al)$ is a non-vanishing rational function of
$\al$, $g_{R_i,\xi^i}(\al)$ is defined by Definition
\ref{D:co.ofrectgularJack}, and the sum runs over sequences of
partitions $\underline{\xi}=(\xi^0,\xi^1,\cdots,\xi^s)$ associated
with
$\underline{\dot{\xi}}=(\dot{\xi}^0,\dot{\xi}^1,\cdots,\dot{\xi}^s)$
such that
$\dot{\xi}^i=\xi^{i+1}\backslash\dot{\xi}^{i+1}\subset'\xi^{i}$
for $i=s-1,\cdots,2,1,0$ with $\dot{\xi}^{s}=0$ and $\xi^0=\xi^1$.
\end{corollary}

\begin{proof} Note that Corollary \ref{C:formula1} implies the expression for
$Q_\la$ with the help of (\ref{F:iterativeJack}).  In the
computation we introduced the variables $\dot{\xi^i}$'s to simplify
the expression and we iteratively use the following identity
$$p_\mu^*.p_\nu=z_\mu\al^{l(\mu)}\binom{m(\nu)}{m(\mu)}p_{\nu\backslash\mu},$$
 which can be checked by the computation:
$(p_i^{m_i})^*.p_i^{n_i}=(i\al)^{m_i}m_i!\binom{n_i}{m_i}p_i^{n_i-m_i}$
for $m_i\leq n_i$.
\end{proof}

\bigskip

\centerline{\bf Acknowledgments}
The first author thanks the support of Universit\"at Basel for hospitality. The second author
gratefully acknowledges the partial support of Max-Planck Institut f\"ur Mathematik in Bonn, Simons Foundation grant 198129, NSFC grant 10728102, and NSF grants 1014554 and 1137837 during this work.

 \vskip 0.1in

\bibliographystyle{amsalpha}

\end{document}